\documentclass[11pt]{article}
\usepackage{enumerate}
\usepackage{lineno,hyperref}
\usepackage{amsfonts}
\usepackage{amssymb}
\usepackage{mathrsfs}
\usepackage{srcltx}
\usepackage{cite}
\usepackage{amsmath}
\usepackage{amsthm}
\usepackage{amstext}
\usepackage{amsopn}
\usepackage{graphicx}
\modulolinenumbers[4]
\allowdisplaybreaks[4]

\newtheorem{theorem}{Theorem}[section]
\newtheorem{lemma}[theorem]{Lemma}
\newtheorem{proposition}[theorem]{Proposition}
\newtheorem{corollary}[theorem]{Corollary}
\newtheorem{remark}[theorem]{Remark}

\theoremstyle{definition}

\numberwithin{equation}{section}

\newcommand{\ba}{\begin{array}}
\newcommand{\ea}{\end{array}}

\newcommand{\ep}{\varepsilon}

\newcommand{\R}{\mathbb{R}}

\newcommand{\N}{\mathbb{N}}

\begin{document}
\date{}
 \title{Singularly perturbed Choquard equations with nonlinearity satisfying Berestycki-Lions assumptions}
\author{Xianhua Tang\footnote{Corresponding author.
E-mail address:  {\tt tangxh@mail.csu.edu.cn} (X. H. Tang), {\tt  mathsitongchen@163.com} (S. T. Chen).},
\ \ Sitong Chen \\
        {\small \it School of Mathematics and Statistics, Central South University,}\\
        {\small \it Changsha, Hunan 410083, P.R.China}}
\maketitle
\begin{abstract}
 In the present paper, we consider the following singularly perturbed problem:
 \begin{equation*}
 \left\{
   \begin{array}{ll}
     -\varepsilon^2\triangle u+V(x)u=\varepsilon^{-\alpha}(I_{\alpha}*F(u))f(u), & x\in \R^N; \\
     u\in H^1(\R^N),
   \end{array}
 \right.
 \end{equation*}
 where $\varepsilon>0$ is a parameter, $N\ge 3$, $\alpha\in (0, N)$, $F(t)=\int_{0}^{t}f(s)\mathrm{d}s$ and $I_{\alpha}: \R^N\rightarrow \R$ is the Riesz potential. By introducing some new tricks, we prove that the above
 problem admits a semiclassical ground state solution ($\ep\in (0,\varepsilon_0)$) and a ground state solution
 ($\varepsilon=1$) under the general ``Berestycki-Lions assumptions" on the nonlinearity $f$ which are almost necessary, as well as  some weak assumptions on the potential $V$. When $\varepsilon=1$, our results generalize and improve the ones in [V. Moroz, J. Van Schaftingen, T. Am. Math. Soc. 367 (2015) 6557-6579] and [H. Berestycki, P.L. Lions,  Arch. Rational Mech. Anal. 82 (1983) 313-345] and some other related literature. In particular, our approach is useful for many similar problems.


\end{abstract}











{\section{Introduction}}
 \setcounter{equation}{0}

  In this paper, we consider the following singularly perturbed nonlinear Choquard equation:
 \begin{equation}\label{KE9}
 \left\{
   \begin{array}{ll}
     -\varepsilon^2\triangle u+V(x)u=\varepsilon^{-\alpha}(I_{\alpha}*F(u))f(u), & x\in \R^N; \\
     u\in H^1(\R^N),
   \end{array}
 \right.
 \end{equation}
 where $\varepsilon>0$ is a parameter, $N\ge 3$, $\alpha\in (0, N)$ and $I_{\alpha}: \R^N\rightarrow \R$ is the Riesz potential defined by
 $$
   I_{\alpha}(x)=\frac{\Gamma\left(\frac{N-\alpha}{2}\right)}{\Gamma\left(\frac{\alpha}{2}\right)
       2^{\alpha}\pi^{N/2}|x|^{N-\alpha}}, \ \ \ \  x \in \R^N\setminus \{0\},
 $$
 $F(t)=\int_{0}^{t}f(s)\mathrm{d}s$, $V: \R^N\rightarrow \R$ and $f: \R \rightarrow \R$ satisfy the following basic assumptions:

 \vskip2mm
 \begin{itemize}
 \item[(V1)] $V\in \mathcal{C}(\R^N, [0, \infty))$ and $V_{\infty}:=\lim_{|x|\to \infty}V(x)>0$;

 \item[(F1)]  $f\in \mathcal{C}(\R, \R)$ and there exists a constant $\mathcal{C}_0>0$ such that
 $$
   |f(t)t|\le \mathcal{C}_0\left(|t|^{(N+\alpha)/N}+|t|^{(N+\alpha)/(N-2)}\right), \ \ \ \ \forall \ t\in \R;
 $$

 \item[(F2)]  $F(t)=o\left(t^{(N+\alpha)/N}\right)$ as $t\to 0$ and $F(t)=o\left(t^{(N+\alpha)/(N-2)}\right)$ as $|t|\to \infty$;

 \item[(F3)] there exists $s_0>0$ such that $F(s_0)\ne 0$.
 \end{itemize}
 Note that (F1)-(F3) were almost necessary and sufficient conditions and regarded as the Berestycki-Lions
 type conditions to Choquard equations, which were introduced by Moroz and Van Schaftingen in \cite{MS}
 for the study of \eqref{KE9} with $\ep=1$.

 \par
 In recent years, semiclassical problems like \eqref{KE9}, i.e. the parameter $\varepsilon$ goes to zero,
 have received attention from the mathematical community. For small $\varepsilon>0$, bound states are called
 semiclassical states, which describe a kind of transition from Quantum Mechanics to Newtonian
 Mechanics. There are some nice work on semiclassical states for \eqref{KE9}. For example,
 for a special form of \eqref{KE9} with $N = 3, \alpha= 2$ and $F(s) = s^2/2$, by proving the uniqueness and
 non-degeneracy, of the ground states for the limit problem, Wei and Winter \cite{WW} constructed a family of solutions by a Lyapunov-Schmidt type reduction; Cingolani et al. \cite{CS} proved the existence of solutions concentrating around several minimum points of $V$ by a global penalization method. Moroz and Van Schaftingen \cite{MS6} developed a nonlocal penalization technique to show that problem \eqref{KE9} with $F(s) = |s|^p/p$ and $p\ge 2$ has a family of solutions concentrating at the local minimum of $V$ provided $V$ satisfies some additional assumptions at infinity. However, for \eqref{KE9} with general nonlinearity $F$ which only satisfies (F1)-(F3), there seem to be no results in the existing literature. One of main purpose of this paper is to deal with this case.

 \par
   When $\varepsilon=1$, \eqref{KE9} reduces to the nonlinear Choquard equation of the form:
 \begin{equation}\label{SE}
 \left\{
   \begin{array}{ll}
     -\triangle u+V(x)u=(I_{\alpha}*F(u))f(u), & x\in \R^N; \\
     u\in H^1(\R^N).
   \end{array}
 \right.
 \end{equation}
 which has been extensively studied by using variational methods, see
 \cite{AG,ANY,AY,CZT,MRZ1, MS1,MS2,MS,MS4,MS6,MS5,RS,TLY} and references therein.
 In view of (F1), (F2) and Hardy-Littlewood-Sobolev inequality,
 for some $p\in (2, 2^*)$ and any $\epsilon>0$, one has
 \begin{eqnarray}\label{Ru}
    &     & \int_{\R^N}(I_{\alpha}*F(u))F(u)\mathrm{d}x\nonumber\\
    &  =  & \frac{\Gamma\left(\frac{N-\alpha}{2}\right)}{\Gamma\left(\frac{\alpha}{2}\right)2^{\alpha}\pi^{N/2}}
              \int_{\R^N}\int_{\R^N}\frac{F(u(x))F(u(y))}{|x-y|^{N-\alpha}}\mathrm{d}x\mathrm{d}y
               \le \mathcal{C}_1\|F(u)\|_{2N/(N+\alpha)}^2\nonumber\\
    & \le &  \epsilon \left(\|u\|_2^{2(N+\alpha)/N}+\|u\|_{2^*}^{2(N+\alpha)/(N-2)}\right)
              +C_{\epsilon}\|u\|_{p}^{(N+\alpha)p/N},  \ \ \ \ \forall \ u\in H^1(\R^N).
 \end{eqnarray}
 It is standard to check using \eqref{Ru} that under (V1), (F1) and (F2), the energy functional
 defined in $H^1(\R^N)$ by
 \begin{equation}\label{IU}
   \mathcal{I}(u) = \frac{1}{2}\int_{\R^N}\left[|\nabla u|^2+V(x)u^2\right]\mathrm{d}x
           -\frac{1}{2}\int_{\R^N}(I_{\alpha}*F(u))F(u)\mathrm{d}x
 \end{equation}
 is continuously differentiable and its critical points correspond to the weak solutions of \eqref{SE}.

 \vskip2mm
 \par
   If the potential $V (x)\equiv V_{\infty}$, then \eqref{SE} reduces to the following autonomous form:
 \begin{equation}\label{SE1}
 \left\{
   \begin{array}{ll}
     -\triangle u+V_{\infty}u=(I_{\alpha}*F(u))f(u), & x\in \R^N; \\
     u\in H^1(\R^N),
   \end{array}
 \right.
 \end{equation}
 its energy functional is as follows:
 \begin{equation}\label{Ii}
   \mathcal{I}^{\infty}(u)=\frac{1}{2}\int_{{\R}^N}\left(|\nabla u|^2+V_{\infty}u^2\right)\mathrm{d}x
            -\frac{1}{2}\int_{\R^N}(I_{\alpha}*F(u))F(u)\mathrm{d}x.
 \end{equation}
 Problem \eqref{SE1} is a semilinear elliptic equation with a nonlocal nonlinearity. For $N = 3, \alpha = 2$, $V_{\infty}=1$ and $F(t) = t^2/2$, it covers in particular the Choquard-Pekar equation
 \begin{equation}\label{SE2}
 \left\{
   \begin{array}{ll}
     -\triangle u+u=(I_2*u^2)u, & x\in \R^3; \\
     u\in H^1(\R^3),
   \end{array}
 \right.
 \end{equation}
 introduced by Pekar \cite{Pe} at least in 1954, describing the quantum mechanics of a polaron at rest.
 In 1976, Choquard \cite{Lie} used \eqref{SE2} to describe an electron trapped in its own hole.
 In 1996, Penrose \cite{MP} proposed \eqref{SE2} as a model of self-gravitating matter.
 In this context \eqref{SE2} is usually called the nonlinear Schr\"odinger-Newton equation,
 see Moroz-Schaftingen \cite{MS}.  

 \par
  If we let $\alpha \rightarrow 0$ in \eqref{SE1}, then we can get the following limiting problem:
 \begin{equation}\label{SE3}
 \left\{
   \begin{array}{ll}
     -\triangle u+V_{\infty}u=g(u), & x\in \R^N; \\
     u\in H^1(\R^N),
   \end{array}
 \right.
 \end{equation}
 where $g=Ff$.
 In the fundamental paper \cite{BL}, Berestycki-Lions proved that \eqref{SE3} has a radially symmetric positive solution provided
 that $g$ satisfies the following assumptions:

 \begin{itemize}
 \item[(G1)]  $g\in \mathcal{C}(\R, \R)$ is odd and there exists a constant $\mathcal{C}_0>0$ such that
 $$
   |g(t)|\le \mathcal{C}_0\left(1+|t|^{(N+2)/(N-2)}\right), \ \ \ \ \forall \ t\in \R;
 $$

 \item[(G2)]  $g(t)=o(t)$ as $t\to 0$ and $g(t)=o\left(t^{(N+2)/(N-2)}\right)$ as $t\to +\infty$;

 \item[(G3)] there exists $s_0>0$ such that $G(s_0)>\frac{1}{2}V_{\infty}s_0^2$, where $G(t)=\int_{0}^{t}g(s)\mathrm{d}s$.
 \end{itemize}

 \par
   To prove the above result, Berestycki-Lions \cite{BL} considered the following constrained
 minimization problem
 \begin{equation}\label{min}
   \min\left\{\|\nabla u\|_2^2 : u\in\mathcal{S} \right\},
 \end{equation}
 where
 \begin{equation}\label{mS}
   \mathcal{S} = \left\{u\in H^1(\R^N) : \int_{\R^N}\left[G(u)-\frac{1}{2}V_{\infty}u^2\right]\mathrm{d}x=1\right\};
 \end{equation}
 they first showed that by the P\'olya-Szeg\"o inequality for the Schwarz symmetrization, the minimum can be taken on radial
 and radially nonincreasing functions. Then they showed the existence of a minimum $\hat{w} \in H^1(\R^N)$ by the direct method of the
 calculus of variations. With the Lagrange multiplier Theorem, they concluded that $\bar{u}(x):=\hat{w}(x/t_{\hat{w}})$ with
 $t_{\hat{w}}=\sqrt{\frac{N-2}{2N}}\|\nabla \hat{w}\|_2$ is a least energy solution of \eqref{SE3}. By noting the one-to-one
 correspondence between $\mathcal{S}$ and
 $$
   \mathcal{P}_G:=\left\{u\in H^1(\R^N)\setminus \{0\} : \frac{N-2}{2}\|\nabla u\|_2^2+\frac{NV_{\infty}}{2}\|u\|_2^2
     -N\int_{\R^N}G(u)\mathrm{d}x=0\right\},
 $$
 Jeanjean-Tanaka \cite{JT} proved that $\bar{u}$ minimizes the value of the energy functional on the Poho\u zaev manifold for \eqref{SE3}.

 \par
  However, the approach of Berestycki-Lions \cite{BL} fails for nonlocal problem \eqref{SE1}
 due to the appearance of the nonlocal term. In \cite{MS}, Moroz-Van Schaftingen proved firstly the existence of a least energy solution to \eqref{SE1} under (F1)-(F3). To do that, they employed a scaling technique introduced by Jeanjean \cite{Je0} to construct a Palais-Smale sequence ((PS)-sequence in short) that satisfies asymptotically the Poho\u zaev identity (a Poho\u zaev-Palais-Smale sequence in short), where the information related to the Poho\u zaev identity helps to ensure the boundedness of (PS)-sequences, and then used a concentration compactness argument to overcome the difficulty caused by lack of Sobolev embeddings. Such an approach could be useful for the study of other problems where radial symmetry of solutions either fails or is not readily available. For more related results on nonlocal problems, we refer to \cite{CT1,MRZ1,MRZ2,MR,SS,XC}.

 \par
   We would like to point out that the approach used in \cite{MS} is only valid for autonomous equations,
 it does not work any more for nonautonomous equation \eqref{SE} with $V\ne$ constant, since one could not construct a Poho\u zaev-Palais-Smale sequence as Moroz-Van Schaftingen did in \cite{MS}. Thus new techniques are required for the study of the nonautonomous equation \eqref{SE} with $f$ satisfying (F1)-(F3), which is another focus of this paper.

 \par
    In view of \cite[Theorem 3]{MS}, every solution $u$ of \eqref{SE1} satisfies the following Poho\u zaev type identity:
 \begin{equation}\label{Pi}
  \mathcal{P}^{\infty}(u):=\frac{N-2}{2}\|\nabla u\|_2^2+\frac{NV_{\infty}}{2}\|u\|_2^2
     -\frac{N+\alpha}{2}\int_{\R^N}(I_{\alpha}*F(u))F(u)\mathrm{d}x=0.
 \end{equation}
 Therefore, the following set
 $$
   \mathcal{M}^{\infty}:=\left\{u\in H^1(\R^N)\setminus \{0\}: \mathcal{P}^{\infty}(u)=0\right\}
 $$
 is a natural constraint for the functional $\mathcal{I}^{\infty}$. Moreover, the least energy solution $u_0$ obtained in \cite{MS}
 satisfies $\mathcal{I}^{\infty}(u_0)\ge \inf_{\mathcal{M}^{\infty}}\mathcal{I}^{\infty}$. A natural question is whether there exists
 a solution $\bar{u}\in \mathcal{M}^{\infty}$ such that
 \begin{equation}\label{Gs}
   \mathcal{I}^{\infty}(\bar{u})=\inf_{\mathcal{M}^{\infty}}\mathcal{I}^{\infty}.
 \end{equation}

 \par
  In the first part of this paper, motivated by \cite{BL,TC3,JT,MS,TL,TC2}, we shall develop a
 more direct approach to obtain a ground state solution for \eqref{SE} which has minimal ``energy"
 $\mathcal{I}$ in the set of all nontrivial solutions, moreover, this solution also minimizes the value of $\mathcal{I}$ on the Poho\u zaev manifold associated with \eqref{SE}, under (F1)-(F3), (V1) and the following
 two additional conditions on $V$:
 \begin{itemize}
 \item[(V2)]  $V(x)\le V_{\infty}$ for all $x\in \R^N$;
 \item[(V3)] $V\in \mathcal{C}^1(\R^N, \R)$ and there exists $\theta\in [0, 1)$ such that
 $t\mapsto \frac{NV(tx)+\nabla V(tx)\cdot (tx)}{t^{\alpha}}+\frac{(N-2)^3\theta }{4t^{\alpha+2}|x|^2}$ is nonincreasing on
 $(0,\infty)$ for all $x\in \R^N\setminus \{0\}$.
 \end{itemize}

 \par
   To state our first result, we define a functional on $H^1(\R^N)$ as follows:
 \begin{eqnarray}\label{Jv}
   \mathcal{P}(u) &  :=  & \frac{N-2}{2}\|\nabla u\|_2^2+\frac{1}{2}\int_{{\R}^N}[NV(x)+\nabla V(x)\cdot x]u^2\mathrm{d}x\nonumber\\
                  &      & \ \ -\frac{N+\alpha}{2}\int_{\R^N}(I_{\alpha}*F(u))F(u)\mathrm{d}x,
 \end{eqnarray}
 which is associated with the Poho\u zaev identity $\mathcal{P}(u)=0$ of \eqref{SE}, see Lemma \ref{lem 4.2}. Let
 \begin{equation}\label{Ne}
   \mathcal{M}:= \left\{u\in H^1(\R^N)\setminus \{0\} : \mathcal{P}(u)=0\right\}.
 \end{equation}

 \par
  Our first result is as follows.

 \begin{theorem}\label{thm1.1}
  Assume that $V$ and $f$ satisfy {\rm (V1)-(V3)} and {\rm (F1)-(F3)}. Then
 problem \eqref{SE} has a solution $\bar{u}\in H^1(\R^N)$ such that $\mathcal{I}(\bar{u})=\inf_{\mathcal{M}}\mathcal{I}
 =\inf_{u\in \Lambda}\max_{t > 0}\mathcal{I}(u_t)>0$, where
 $$
   u_t(x):=u(x/t) \ \ \mbox{and} \ \
  \Lambda :=\left\{u\in H^1(\R^N) : \int_{\R^N}(I_{\alpha}*F(u))F(u)\mathrm{d}x > 0\right\}.
 $$
 \end{theorem}

 \begin{corollary}\label{cor1.2}
   Assume that $f$ satisfies {\rm (F1)-(F3)}. Then problem \eqref{SE1} has a solution
 $\bar{u}\in H^1(\R^N)$ such that $\mathcal{I}^{\infty}(\bar{u})=\inf_{\mathcal{M}^{\infty}}\mathcal{I}^{\infty}
 =\inf_{u\in \Lambda}\max_{t > 0}\mathcal{I}^{\infty}(u_t)>0$.
 \end{corollary}

 \par
   With the help of the Poho\u zaev type identity \eqref{Pi} established in \cite{MS}, we easily prove that the solution $\bar{u}$ obtained
 in Corollary \eqref{cor1.2} is also the least energy solution for \eqref{SE1}. More precisely, we have the following theorem:

 \begin{theorem}\label{thm1.3}
   Assume that $f$ satisfies {\rm (F1)-(F3)}. Then problem \eqref{SE1} has a solution $\bar{u}\in H^1(\R^N)$ such that
 $$
   \mathcal{I}^{\infty}(\bar{u})=\inf_{\mathcal{M}^{\infty}}\mathcal{I}^{\infty} =\inf\left\{\mathcal{I}^{\infty}(u): u\in H^1(\R^N)\setminus \{0\}
      \ \mbox{is a solution of} \ \eqref{SE1}\right\}.
 $$
 \end{theorem}

 \begin{remark}\label{rem1.4}
 {\rm (V3)} is a mild condition. In fact, $V$ satisfies {\rm (V3)}  if the following assumption holds:
 \begin{itemize}
 \item[{\rm(V3$'$)}] $V\in \mathcal{C}^1(\R^N, \R)$ and $t\mapsto \frac{NV(tx)+\nabla V(tx)\cdot (tx)}{t^{\alpha}}$ is nonincreasing on
 $(0,\infty)$ for all $x\in \R^N$.
 \end{itemize}
 There are indeed many functions which satisfy {\rm (V1)-(V3)}. For example
 \par
 {\rm i)}. $V(x) =a-\frac{b}{|x|^2+1}$ with $a>b$ and $\alpha Na+(\alpha+2)(N-2)^3>[(N-2)(\alpha+2)+2(\alpha+4)]b>0$;
 \par
 {\rm ii)}. $V(x) =a-\frac{b}{|x|^{\alpha}+1}$ with $a\ge (2+\alpha/N)b>0$;
 \par
 {\rm iii)}. $V(x) =a-be^{-|x|^{\alpha}}$ with $a>b>0$.
 \end{remark}

 \begin{remark}\label{rem1.5}
  We point out that, as a consequence of Theorem \ref{thm1.1}, the least energy value $m:=\inf_{\mathcal{M}}\mathcal{I}$
 has a minimax  characterization $ m=\inf_{u\in \Lambda}\max_{t > 0}\mathcal{I}(u_t)$ which is much simpler than the usual characterizations
 related to the Mountain Pass level.
 \end{remark}

  \par
    In the second part of this paper, we are interested in the existence of the least energy solutions
  for \eqref{SE} under (F1)-(F3). In this case, we can replace (V3) by the following weaker decay
  assumption on $\nabla V$:
 \begin{itemize}
 \item[(V4)] $V\in \mathcal{C}^1(\R^N, \R)$ and there exist $\theta'\in (0,1)$ and $\bar{R}\ge 0$ such that
 $$
   \nabla V(x)\cdot x\le
   \left\{\begin{array}{ll}\frac{(N-2)^2}{2|x|^2}, \ \ & 0<|x| < \bar{R};\\
     \theta'\alpha V(x), \ \ & |x|\ge \bar{R}.
   \end{array}\right.
 $$
 \end{itemize}
 In this direction, we have the following theorem.

 \begin{theorem}\label{thm1.6}
  Assume that $V$ and $f$ satisfy {\rm (V1), (V2), (V4)}  and {\rm (F1)-(F3)}. Then problem \eqref{SE}
 has a solution $\bar{u}\in H^1(\R^N)$ such that $\mathcal{I}(\bar{u}) =\inf_{\mathcal{K}}\mathcal{I}$, where
 $$
   \mathcal{K}:=\left\{u\in H^1(\R^N)\setminus \{0\} : \mathcal{I}'(u)=0\right\}.
 $$
 \end{theorem}

 \begin{remark}\label{rem1.7}
  {\rm (V1), (V2)}  and {\rm (V4)}  are satisfied by a very wide class of potentials. For example,
 $V(x)=a-\frac{b}{1+|x|^{\beta}}$ satisfies {\rm (V1)}  and {\rm (V4)} for $\beta>0$ and $\alpha a>
 (\alpha+\beta)b>0$.
 \end{remark}

 \vskip4mm
 \par
   Applying Theorem \ref{thm1.6} to the following perturbed problem:
 \begin{equation}\label{KE8}
 \left\{
   \begin{array}{ll}
     -\triangle u+[V_{\infty}-\varepsilon h(x)]u=(I_{\alpha}*F(u))f(u), & x\in \R^N; \\
     u\in H^1(\R^N),
   \end{array}
 \right.
 \end{equation}
 where $V_{\infty}$ is a positive constant and the function $h \in \mathcal{C}^1(\R^N, \R)$ verifies:

 \vskip2mm
 \noindent
 (H1)\ \ $h(x) \ge 0$ for all $x\in \R^N$ and $\lim_{|x|\to\infty}h(x)=0$;

 \vskip2mm
 \noindent
 (H2)\ \ $\sup_{x\in \R^N}\left[-\nabla h(x)\cdot x\right]<\infty$.

 \vskip2mm
 \noindent
 Then we have the following corollary.

 \begin{corollary}\label{cor1.8}
   Assume that $h$ and $f$ satisfy {\rm (H1), (H2)}  and {\rm (F1)-(F3)}.  Then there
 exists a constant $\hat{\varepsilon}>0$ such that problem \eqref{KE8} has a least energy solution $\bar{u}_{\varepsilon}\in
 H^1(\R^N)\setminus \{0\}$ for all $0<\varepsilon\le \hat{\varepsilon}$.
 \end{corollary}

 \par
   In the last part of the present paper, we consider the singularly perturbed nonlinear Choquard equation
 \eqref{KE9}, and prove the existence of semiclassical ground state solutions for \eqref{KE9}
 under weaker assumptions on $V$ :
 \begin{itemize}
 \item[(V5)]  $0<V(x_0):=\min_{x\in\R^N}V(x)< V_{\infty}$ for some $x_0\in \R^N$;

 \item[(V6)] $V\in \mathcal{C}^1(\R^N, \R)$ and there exists $\theta''\in (0,1)$ such that
 $$
   \nabla V(x)\cdot x\le \theta''\alpha V(x), \ \ \ \ \forall \ x\in \R^N.
 $$
 \end{itemize}
 Condition (V5) was introduced by Rabinowitz in \cite{Ra}. Our last result is as follows.

 \begin{theorem}\label{thm1.9}
  Assume that $V$ and $f$ satisfy {\rm (V5), (V6)}  and {\rm (F1)-(F3)}. Then there
 exists a constant $\varepsilon_0>0$ determined by terms of $N,V$ and $F$ {\rm (}see Lemma \ref{lem 5.2}{\rm )} such that
 problem \eqref{KE9} has a least energy solution $\bar{u}_{\varepsilon}\in
 H^1(\R^N)\setminus \{0\}$ for all $0<\varepsilon\le \varepsilon_0$.
 \end{theorem}

 \begin{remark}\label{rem1.10}
  {\rm (V5)} is weaker than {\rm (V2)}. There are many functions that satisfy {\rm (V5)}  and {\rm (V6)}
 but do not satisfy {\rm (V2)}. For example,
 $V(x)=a-\frac{b\cos |x|^{\beta}}{1+|x|^{\beta}}$ satisfies {\rm (V1), (V5)} and {\rm (V6)}
 for $\beta>0$ and $\alpha a> (\alpha+\beta)b>0$.
 \end{remark}

  \begin{remark}\label{rem1.11}
  Our approach could be applied to deal with many equations,
 such as Schr\"odinger equations, see {\rm\cite{TC3}}. In the existing literature, Schr\"odinger equations were considered by many authors {\rm(for example \cite{BL,CWY,GLW,Je0,JT} )}.
 \end{remark}

 \par
   To prove Theorem \ref{thm1.1}, we shall divide our arguments into three steps:
 i). Choosing a minimizing sequence $\{u_n\}$ of $\mathcal{I}$ on $\mathcal{M}$, which satisfies
 \begin{equation}\label{A01}
    \mathcal{I}(u_n)\rightarrow m:=\inf_{\mathcal{M}}\mathcal{I}, \ \ \ \ \ \mathcal{P}(u_n)= 0.
 \end{equation}
 Then showing that $\{u_n\}$ is bounded in $H^1(\R^N)$. ii). With a concentration-compactness argument and ``the least energy
 squeeze approach", showing that $\{u_n\}$ converges weakly to some $\bar{u}\in H^1(\R^N)\setminus \{0\}$. And then
 showing that $\bar{u}\in \mathcal{M}$ and $\mathcal{I}(\bar{u})=\inf_{\mathcal{M}}\mathcal{I}$.
 iii). Showing that $\bar{u}$ is a critical point of $\mathcal{I}$. Of them, Step ii) is the most difficult due to
 lack of global compactness and adequate information on $\mathcal{I}'(u_n)$. To avoid relying radial compactness,
 we establish a crucial inequality related to $\mathcal{I}(u)$, $\mathcal{I}(u_t)$ and $\mathcal{P}(u)$ (Lemma \ref{lem 2.2}), it plays
 a crucial role in our arguments, see Lemmas \ref{lem 2.7}, \ref{lem 2.11}, \ref{lem 3.2}, \ref{lem 4.5}, \ref{lem 5.2}. With the
 help of this inequality, we then can complete Step ii) by using Lions' concentration compactness, the least energy squeeze approach
 and some subtle analysis. Moreover, such an approach could be useful for the study of other problems where radial symmetry of bounded sequence either fails or is not readily available.

 Classically, in order to show the existence of solutions for \eqref{SE}, one compares the critical level with the one of \eqref{SE1}
 (i.e. the problem at infinity). To this end, it is necessary to establish a strict inequality similar to
 $$
   \max_{t\in [0,1]}\mathcal{I}(\gamma_0(t))<\inf\left\{\mathcal{I}^{\infty}(u): u\in H^1(\R^N)\setminus \{0\}
   \ \mbox{is a solution of}\ \eqref{SE1}\right\}
 $$
 for some path $\gamma_0\in \mathcal{C}([0,1], H^1(\R^N))$. Clearly, $\gamma_0(t)>0$ is a natural requirement under (V1),
 which usually involves an additional assumption on $f$ besides (F1)-(F3), such as $f(t)$ is odd and $f(t)t\ge 0$, see \cite[Theorem 1.4]{MS}.
 We would like to point out that the above strict inequality is not used in our arguments, see Section 2. Our approach could be useful for the
 study of other problems where paths or the ground state solutions of the problem at infinity are not sign definite.

 \vskip4mm
 \par
   To prove Theorem \ref{thm1.6}, as in Jeanjean-Tanaka \cite{JT}, for $\lambda\in [1/2, 1]$ we consider the family of functionals $\mathcal{I}_{\lambda} : H^1(\R^N) \rightarrow \R$ defined by
 \begin{equation}\label{Ilu}
   \mathcal{I}_{\lambda}(u)=\frac{1}{2}\int_{\R^N}\left(|\nabla u|^2+V(x)u^2\right)\mathrm{d}x
            -\frac{\lambda}{2}\int_{\R^N}(I_{\alpha}*F(u))F(u)\mathrm{d}x.
 \end{equation}
 These functionals have a Mountain Pass geometry, and denoting $c_{\lambda}$ the corresponding Mountain Pass levels.
 Corresponding to \eqref{Ilu}, we also let
 \begin{equation}\label{Il}
   \mathcal{I}_{\lambda}^{\infty}(u)=\frac{1}{2}\int_{\R^N}\left(|\nabla u|^2+V_{\infty}u^2\right)\mathrm{d}x
       -\frac{\lambda}{2}\int_{\R^N}(I_{\alpha}*F(u))F(u)\mathrm{d}x.
 \end{equation}
 By Corollary \ref{cor1.2}, for every $\lambda\in [1/2, 1]$, there exists a minimizer $u_{\lambda}^{\infty}$ of $\mathcal{I}_{\lambda}^{\infty}$
 on $\mathcal{M}_{\lambda}^{\infty}$, where
 \begin{eqnarray}\label{Ml}
   \mathcal{M}_{\lambda}^{\infty}:=\left\{u\in H^1(\R^N)\setminus \{0\}: \mathcal{P}_{\lambda}^{\infty}(u)=0\right\}
 \end{eqnarray}
 and
 \begin{eqnarray}\label{JlL}
  \mathcal{P}_{\lambda}^{\infty}(u)
    &  =  & \frac{N-2}{2}\|\nabla u\|_2^2+NV_{\infty}\|u\|_2^2-\frac{(N+\alpha)\lambda}{2}\int_{\R^N}(I_{\alpha}*F(u))F(u)\mathrm{d}x.
 \end{eqnarray}
 Let
 $$
   A(u)=\frac{1}{2}\int_{\R^N}\left(|\nabla u|^2+V(x)u^2\right)\mathrm{d}x, \ \ \ \
   B(u)=\frac{1}{2}\int_{\R^N}(I_{\alpha}*F(u))F(u)\mathrm{d}x.
 $$
 Then $\mathcal{I}_{\lambda}(u)=A(u)-\lambda B(u)$. Since $B(u)$ is not sign definite, it prevents us from employing Jeanjean's
 monotonicity trick \cite{Je}. More trouble, it is difficult to show the following key inequality
 \begin{eqnarray}\label{cm0}
   c_{\lambda}<m_{\lambda}^{\infty} :=\inf_{u\in \mathcal{M}_{\lambda}^{\infty}}\mathcal{I}_{\lambda}^{\infty}(u)
     \left(=\mathcal{I}_{\lambda}^{\infty}(u_{\lambda}^{\infty})\right),  \ \ \ \ \lambda\in [1/2, 1]
 \end{eqnarray}
 due to the minimizer $u_{\lambda}^{\infty}$ being not positive definite.

 \par
   Thanks to the work of Jeanjean-Toland \cite{JTo}, $\mathcal{I}_{\lambda}$ still has a bounded (PS)-sequence
 $\{u_n(\lambda)\} \subset H^1(\R^N)$ at level $c_{\lambda}$ for almost every $\lambda\in [1/2,1]$. Different from the arguments in the existing literature, by means of $u_1^{\infty}$ and the key inequality established in Lemma \ref{lem 2.2}, we can find a constant $\bar{\lambda}\in [1/2, 1)$ and then prove directly the following inequality
 \begin{eqnarray}\label{cm1}
   c_{\lambda}<m_{\lambda}^{\infty},\ \ \ \ \lambda\in (\bar{\lambda}, 1],
 \end{eqnarray}
 see Lemma \ref{lem 4.5}. In particular, it is not require any information on sign of $u_1^{\infty}$ in our arguments. Applying \eqref{cm1} and a precise decomposition of bounded (PS)-sequences in \cite{JT}, we can get a nontrivial critical point $u_{\lambda}$ of $\mathcal{I}_{\lambda}$ which possesses energy $c_{\lambda}$  for almost every $\lambda\in [\bar{\lambda}, 1]$. Finally, with a Poho\u zaev identity  we proved that \eqref{SE} admits a least energy solution under (V1), (V2), (V4) and (F1)-(F3).

 \vskip4mm
 \par
   Throughout the paper we make use of the following notations:

 \vskip4mm
 \par
     $\spadesuit$ \ $H^1(\R^N)$ denotes the usual Sobolev space equipped with the inner product and norm
 $$
   (u, v)=\int_{\R^N}(\nabla u\cdot \nabla v+uv)\mathrm{d}x, \ \ \|u\|=(u, u)^{1/2},
     \ \ \forall \ u,v\in H^1(\R^N);
 $$

 \par
     $\spadesuit$ \ $L^s(\R^N) (1\le s< \infty)$  denotes the Lebesgue space with the norm $\|u\|_s
 =\left(\int_{\R^N}|u|^s\mathrm{d}x\right)^{1/s}$;

 \par
     $\spadesuit$ \ For any $u\in H^1(\R^N)\setminus \{0\}$, $u_t(x):=u(t^{-1}x)$ for $t>0$;

 \par
     $\spadesuit$ \ For any $x\in \R^N$ and $r>0$, $B_r(x):=\{y\in \R^N: |y-x|<r \}$;

 \par
     $\spadesuit$ \ $C_1, C_2,\cdots$ denote positive constants possibly different in different places.

 \vskip4mm
   The rest of the paper is organized as follows. In Section 2, we give some preliminaries, and give the proofs
 of Theorems \ref{thm1.1} and \ref{thm1.3}. Section 3 is devoted to finding a least energy solution for \eqref{SE}
 and Theorem \ref{thm1.6} will be proved in this section. In the last section, we show the existence of semiclassical
 ground state solutions for \eqref{KE9} and prove Theorem \ref{thm1.9}.

 \vskip6mm

 {\section{Ground state solutions for \eqref{SE}}}
 \setcounter{equation}{0}

 \setcounter{equation}{0}

 \vskip2mm
 \par
   In this section, we give the proofs of Theorems \ref{thm1.1} and \ref{thm1.3}. To this end, we give some useful lemmas. Since $V(x)\equiv V_{\infty}$
 satisfies (V1)-(V3), thus all conclusions on $\mathcal{I}$ are also true for $\mathcal{I}^{\infty}$.  For \eqref{SE1}, we always assume that $V_{\infty}>0$.
 First, by a simple calculation, we can verify Lemma \ref{lem 2.1}.

 \begin{lemma}\label{lem 2.1}
 The following two inequalities hold:
 \begin{equation}\label{B10}
  \mathfrak{g}(t):=2+\alpha-(N+\alpha)t^{N-2}+(N-2)t^{N+\alpha} > \mathfrak{g}(1)=0,   \ \ \ \ \forall \ t\in [0, 1)\cup(1, +\infty),
 \end{equation}
 \begin{equation}\label{B20}
  \mathfrak{h}(t):=\alpha-(N+\alpha)t^{N}+Nt^{N+\alpha} > \mathfrak{h}(1)=0,   \ \ \ \ \forall \ t\in [0, 1)\cup(1, +\infty).
 \end{equation}
 Moreover {\rm (V3)} implies the following inequality holds:
 \begin{eqnarray}\label{V3}
    &     & \left(\alpha+Nt^{N+\alpha}\right)V(x)-(N+\alpha)t^NV(tx)
             +\left(t^{N+\alpha}-1\right)\nabla V(x)\cdot x \nonumber\\
    & \ge & -\frac{(N-2)^2\theta\left[2+\alpha-(N+\alpha)t^{N-2}+(N-2)t^{N+\alpha}\right] }{4|x|^2},
              \ \ \forall \ t\ge 0, \ \ x\in \R^N\setminus \{0\}.
 \end{eqnarray}
 \end{lemma}

 \begin{lemma}\label{lem 2.0}
   Assume that {\rm (V1)-(V3)} hold. Then
 \begin{equation}\label{L20}
    |\nabla V(x)\cdot x|\to 0\ \  \mbox{as}\ |x|\to \infty.
 \end{equation}
 \end{lemma}

 \begin{proof}
 Arguing by contradiction, we assume that there exist $\{x_n\}\subset \R^N$ and $\delta>0$ such that
 \begin{equation}\label{201}
  |x_n|\to \infty, \ \ \mbox{and} \ \ \nabla V(x_n)\cdot x_n\ge \delta\ \mbox{or}\
   \nabla V(x_n)\cdot x_n\le -\delta
  \ \ \ \ \forall\ n\in \N.
 \end{equation}
 Now, we distinguish two case: i) $\nabla V(x_n)\cdot x_n\ge \delta, \forall\ n\in \N$ and
  ii) $\nabla V(x_n)\cdot x_n\le -\delta, \forall\ n\in \N$.

 Case i) $\nabla V(x_n)\cdot x_n\ge \delta, \forall\ n\in \N$. In this case, by \eqref{V3}, one has
 \begin{eqnarray}\label{201}
  \delta
    & \le & \nabla V(x_n)\cdot x_n\nonumber\\
    & \le & \frac{\left(\alpha+Nt^{N+\alpha}\right)V(x_n)-(N+\alpha)t^NV(tx_n)}{1-t^{N+\alpha}}
             +\frac{(N-2)^2\theta\mathfrak{g}(t)}{4(1-t^{N+\alpha})|x_n|^2},
              \ \ \forall \ 0<t<1.
 \end{eqnarray}
 Since
 \begin{equation}\label{202}
   \lim_{|t|\to 1}\frac{\alpha+Nt^{N+\alpha}-(N+\alpha)t^N}{1-t^{N+\alpha}}=0,
 \end{equation}
 there exists $t_1\in (0,1)$ such that
 \begin{equation}\label{203}
   \frac{\left[\alpha+Nt_1^{N+\alpha}-(N+\alpha)t_1^N\right]V_{\infty}}{1-t_1^{N+\alpha}}<\frac{\delta}{2}.
 \end{equation}
 Then it follows from (V2), \eqref{201} and \eqref{203} that
 \begin{eqnarray}\label{204}
  \delta
    & \le & \frac{\left[\alpha+Nt_1^{N+\alpha}-(N+\alpha)t_1^N\right]V(x_n)}{1-t_1^{N+\alpha}}
             +\frac{(N+\alpha)t_1^N}{1-t_1^{N+\alpha}}[V(x_n)-V(t_1x_n)]\nonumber\\
    &     & \ \  +\frac{(N-2)^2\theta\mathfrak{g}(t_1)}{4(1-t_1^{N+\alpha})|x_n|^2}\nonumber\\
    & \le & \frac{\delta}{2}+\frac{(N+\alpha)t_1^N}{1-t_1^{N+\alpha}}[V(x_n)-V(t_1x_n)]
              +\frac{(N-2)^2\theta\mathfrak{g}(t_1)}{4(1-t_1^{N+\alpha})|x_n|^2}\nonumber\\
    &  =  & \frac{\delta}{2}+o(1),
 \end{eqnarray}
 which is a contradiction.

  Case ii) $\nabla V(x_n)\cdot x_n\le -\delta, \forall\ n\in \N$. In this case, by \eqref{V3}, one has
 \begin{eqnarray}\label{205}
  -\delta
    & \ge & \nabla V(x_n)\cdot x_n\nonumber\\
    & \ge & \frac{(N+\alpha)t^NV(tx_n)-\left(\alpha+Nt^{N+\alpha}\right)V(x_n)}{t^{N+\alpha}-1}
             -\frac{(N-2)^2\theta\mathfrak{g}(t)}{4(t^{N+\alpha}-1)|x_n|^2},
              \ \ \forall \ t>1.
 \end{eqnarray}
 From \eqref{202}, there exists $t_2>1$ such that
 \begin{equation}\label{206}
\frac{\left[(N+\alpha)t_2^N-\alpha-Nt_2^{N+\alpha}\right]V_{\infty}}{t_2^{N+\alpha}-1}>-\frac{\delta}{2}.
 \end{equation}
 Then it follows from (V2), \eqref{205} and \eqref{206} that
 \begin{eqnarray}\label{207}
  -\delta
    & \ge & \frac{\left[(N+\alpha)t_2^N-\alpha-Nt_2^{N+\alpha}\right]V(x_n)}{t_2^{N+\alpha}-1}
             +\frac{(N+\alpha)t_2^N}{t_2^{N+\alpha}-1}[V(t_2x_n)-V(x_n)]\nonumber\\
    &     & \ \  -\frac{(N-2)^2\theta\mathfrak{g}(t_2)}{4(t_2^{N+\alpha}-1)|x_n|^2}\nonumber\\
    & \ge & -\frac{\delta}{2}+\frac{(N+\alpha)t_2^N}{t_2^{N+\alpha}-1}[V(t_2x_n)-V(x_n)]
              -\frac{(N-2)^2\theta\mathfrak{g}(t_2)}{4(t_2^{N+\alpha}-1)|x_n|^2}\nonumber\\
    &  =  & -\frac{\delta}{2}+o(1),
 \end{eqnarray}
 which is a contradiction.
 \end{proof}

 \begin{lemma}\label{lem 2.2}
   Assume that {\rm (V1)-(V3), (F1)} and {\rm (F2)} hold. Then
 \begin{eqnarray}\label{B21}
   \mathcal{I}(u)
        & \ge & \mathcal{I}(u_t)+\frac{1-t^{N+\alpha}}{N+\alpha}\mathcal{P}(u)
                 +\frac{(1-\theta)\mathfrak{g}(t)}{2(N+\alpha)}\|\nabla u\|_2^2,\ \ \ \ \forall \ u\in H^1(\R^N), \ \ t > 0.
 \end{eqnarray}
 \end{lemma}

 \begin{proof} According to Hardy inequality, we have
 \begin{equation}\label{B22}
  \|\nabla u\|_2^2 \ge \frac{(N-2)^2}{4}\int_{\R^N}\frac{u^2}{|x|^2}\mathrm{d}x,  \ \ \forall \ u\in H^1(\R^N).
 \end{equation}
 Note that
 \begin{equation}\label{B23}
   \mathcal{I}(u_t) = \frac{t^{N-2}}{2}\|\nabla u\|_2^2+\frac{t^N}{2}\int_{{\R}^N}V(tx)u^2\mathrm{d}x
             -\frac{t^{N+\alpha}}{2}\int_{\R^N}(I_{\alpha}*F(u))F(u)\mathrm{d}x.
 \end{equation}
 Thus, by \eqref{IU}, \eqref{Jv}, \eqref{B10}, \eqref{V3}, \eqref{B22} and \eqref{B23}, one has
 \begin{eqnarray*}
   &     & \mathcal{I}(u)-\mathcal{I}(u_t)\\
   &  =  & \frac{1-t^{N-2}}{2}\|\nabla u\|_2^2+\frac{1}{2}\int_{{\R}^N}\left[V(x)-t^{N}V(tx)\right]u^2\mathrm{d}x\\
   &     & \ \    -\frac{1-t^{N+\alpha}}{2}\int_{\R^N}(I_{\alpha}*F(u))F(u)\mathrm{d}x\\
   &  =  & \frac{1-t^{N+\alpha}}{N+\alpha}\left\{\frac{N-2}{2}\|\nabla u\|_2^2
             +\frac{1}{2}\int_{\R^N}[NV(x)+\nabla V(x)\cdot x]u^2\mathrm{d}x\right.\\
   &     & \ \ \left.     -\frac{N+\alpha}{2}\int_{\R^N}(I_{\alpha}*F(u))F(u)\mathrm{d}x\right\}\\
   &     & +\frac{2+\alpha-(N+\alpha)t^{N-2}+(N-2)t^{N+\alpha}}{2(N+\alpha)}\|\nabla u\|_2^2\\
   &     &   +\frac{1}{2}\int_{\R^N}\left\{\left[\frac{\alpha+Nt^{N+\alpha}}{N+\alpha}V(x)-t^NV(tx)\right]
             -\frac{1-t^{N+\alpha}}{N+\alpha}\nabla V(x)\cdot x\right\}u^2\mathrm{d}x\\
   & \ge & \frac{1-t^{N+\alpha}}{N+\alpha}\mathcal{P}(u)+\frac{(1-\theta)\mathfrak{g}(t)}{2(N+\alpha)}\|\nabla u\|_2^2.
 \end{eqnarray*}
 This shows that \eqref{B21} holds.
 \end{proof}

   From Lemma \ref{lem 2.2}, we have the following two corollaries.

 \begin{corollary}\label{cor2.3}
   Assume that {\rm (F1)}  and {\rm (F2)} hold. Then
 \begin{eqnarray}\label{B25}
   \mathcal{I}^{\infty}(u)
     &  =  & \mathcal{I}^{\infty}(u_t)+\frac{1-t^{N+\alpha}}{N+\alpha}\mathcal{P}^{\infty}(u)
               +\frac{\mathfrak{g}(t)\|\nabla u\|_2^2+V_{\infty}\mathfrak{h}(t)\|u\|_2^2}{2(N+\alpha)},\nonumber\\
     &     & \ \ \ \ \ \ \ \ \forall \ u\in H^1(\R^N), \ \ t > 0.
 \end{eqnarray}
 \end{corollary}

 \begin{corollary}\label{cor2.4}
   Assume that {\rm (V1)-(V3), (F1)}  and  {\rm (F2)} hold. Then for $u\in \mathcal{M}$
 \begin{equation}\label{Imax}
   \mathcal{I}(u) = \max_{t> 0}\mathcal{I}(u_t).
 \end{equation}
 \end{corollary}

 \begin{lemma}\label{lem 2.5}
   Assume that {\rm (V1)-(V3)} hold. Then there exist two constants $\gamma_1, \gamma_2>0$ such that
 \begin{equation}\label{B26}
   \gamma_1\|u\|^2\le (N-2)\|\nabla u\|_2^2+\int_{\R^N}\left[NV(x)+\nabla V(x)\cdot x\right]u^2\mathrm{d}x\le \gamma_2\|u\|^2,
     \ \ \forall \ u\in H^1(\R^N).
 \end{equation}
 \end{lemma}

 \begin{proof} Let $t=0$ and $t \to \infty$ in \eqref{V3}, respectively, and using (V1), (V2), one has
 \begin{equation}\label{B27}
   \nabla V(x)\cdot x\le \alpha V_{\infty}+\frac{(N-2)^2(2+\alpha)\theta}{4|x|^2},
       \ \ \ \ \forall \ x\in \R^N\setminus \{0\},
 \end{equation}
 \begin{equation}\label{B28}
   -NV_{\infty}-\frac{(N-2)^3\theta}{4|x|^2}\le -NV(x)-\frac{(N-2)^3\theta}{4|x|^2}
     \le \nabla V(x)\cdot x, \ \ \ \ \forall \ x\in \R^N\setminus \{0\}.
 \end{equation}
 By \eqref{B27}, \eqref{B28} and $V\in \mathcal{C}^1(\R^N, \R)$, there exists a constant $M_0>0$ such that
 \begin{equation}\label{B29}
   |\nabla V(x)\cdot x| \le M_0, \ \ \ \ \forall \ x\in \R^N.
 \end{equation}
 From \eqref{V3}, one has
 \begin{eqnarray}\label{B30}
    &     & NV(x)+\nabla V(x)\cdot x \nonumber\\
    & \ge & -\frac{(N-2)^3\theta}{4|x|^2}+(N+\alpha)t^{-\alpha}V(tx)\nonumber\\
    &     & -\left[\frac{(N-2)^2(2+\alpha)\theta}{4|x|^2}-\nabla V(x)\cdot x+\alpha V(x)\right]t^{-N-\alpha},
              \ \ \ \ \forall \ t > 0, \ \ x\in \R^N\setminus \{0\}. \ \ \ \
 \end{eqnarray}
 By (V1), there exists $R>0$ such that $V(x)\ge \frac{V_{\infty}}{2}$ for all $|x|\ge R$ and
 \begin{equation}\label{B31}
   \left[\frac{(N-2)^2(2+\alpha)\theta}{4}+M_0+\alpha V_{\infty}\right]R^{-N}<\frac{(N+\alpha)V_{\infty}}{4}.
 \end{equation}
 It follows from (V1), (V2), \eqref{B29}, \eqref{B30} and \eqref{B31} that
 \begin{eqnarray}\label{B32}
   NV(x)+\nabla V(x)\cdot x
    & \ge & -\frac{(N-2)^3\theta}{4|x|^2}+(N+\alpha)R^{-\alpha}V(Rx)\nonumber\\
    &     & -\left[\frac{(N-2)^2(2+\alpha)\theta}{4|x|^2}-\nabla V(x)\cdot x+\alpha V(x)\right]R^{-N-\alpha}\nonumber\\
    & \ge & -\frac{(N-2)^3\theta}{4|x|^2}+\frac{(N+\alpha)R^{-\alpha}V_{\infty}}{4},
              \ \ \ \ \forall \ |x|\ge 1.
 \end{eqnarray}
 Making use of the H\"older inequality and the Sobolev inequality, we get
 \begin{equation}\label{B33}
   \int_{|x|< 1}u^2\mathrm{d}x\le \omega_N^{(2^*-2)/2^*}\left(\int_{|x|< 1}|u|^{2^*}\mathrm{d}x\right)^{2/2^*}
     \le \omega_N^{2/N}S^{-1}\|\nabla u\|_2^2,
 \end{equation}
 where $\omega_N$ denotes the volume of the unit ball of $\R^N$. Thus it follows from \eqref{B22}, \eqref{B27}, \eqref{B28}, \eqref{B32}
 and \eqref{B33} that
 \begin{eqnarray}\label{B34}
   &     & (N-2)\|\nabla u\|_2^2+\int_{\R^N}\left[NV(x)+\nabla V(x)\cdot x\right]u^2\mathrm{d}x\nonumber\\
   & \le & [N-2+(2+\alpha)\theta]\|\nabla u\|_2^2+(N+\alpha)V_{\infty}\|u\|_2^2\nonumber\\
   & \le & [N-2+(2+\alpha)\theta+(N+\alpha)V_{\infty}]\|u\|^2:=\gamma_2\|u\|^2, \ \ \forall \ u\in H^1(\R^N)
 \end{eqnarray}
 and
 \begin{eqnarray}\label{B35}
   &     & (N-2)\|\nabla u\|_2^2+\int_{\R^N}\left[NV(x)+\nabla V(x)\cdot x\right]u^2\mathrm{d}x\nonumber\\
   &  =  & (N-2)\|\nabla u\|_2^2+\int_{|x|<1}\left[NV(x)+\nabla V(x)\cdot x\right]u^2\mathrm{d}x
             +\int_{|x|\ge 1}\left[NV(x)+\nabla V(x)\cdot x\right]u^2\mathrm{d}x\nonumber\\
   & \ge & (N-2)\|\nabla u\|_2^2-\frac{(N-2)^3\theta}{4}\int_{\R^N}\frac{u^2}{|x|^2}\mathrm{d}x
             +\frac{(N+\alpha)R^{-\alpha}V_{\infty}}{4}\int_{|x|\ge 1}u^2\mathrm{d}x\nonumber\\
   & \ge & (1-\theta)(N-2)\|\nabla u\|_2^2+\frac{(N+\alpha)R^{-\alpha}V_{\infty}}{4}\int_{|x|\ge 1}u^2\mathrm{d}x\nonumber\\
   & \ge & \frac{(1-\theta)(N-2)}{2}\|\nabla u\|_2^2+\frac{(1-\theta)(N-2)S}{2\omega_N^{2/N}}
            \int_{|x|< 1}u^2\mathrm{d}x+\frac{(N+\alpha)R^{-\alpha}V_{\infty}}{4}\int_{|x|\ge 1}u^2\mathrm{d}x\nonumber\\
   & \ge & \frac{(1-\theta)(N-2)}{2}\|\nabla u\|_2^2+\min\left\{\frac{(1-\theta)(N-2)S}{2\omega_N^{2/N}},
            \frac{(N+\alpha)R^{-\alpha}V_{\infty}}{4}\right\}\|u\|_2^2\nonumber\\
   & \ge & \min\left\{\frac{(1-\theta)(N-2)}{2}, \frac{(1-\theta)(N-2)S}{2\omega_N^{2/N}},
            \frac{(N+\alpha)R^{-\alpha}V_{\infty}}{4}\right\}\|u\|^2\nonumber\\
   & :=  & \gamma_1\|u\|^2, \ \ \forall \ u\in H^1(\R^N).
 \end{eqnarray}
 Both \eqref{B34} and \eqref{B35} imply that \eqref{B26} holds.
 \end{proof}

 \par
   To show $\mathcal{M}\ne \emptyset$, we define a set $\Lambda$ as follows:
 \begin{equation}\label{La}
   \Lambda：=\left\{u\in H^1(\R^N) : \int_{\R^N}(I_{\alpha}*F(u))F(u)\mathrm{d}x > 0\right\}.
 \end{equation}

 \begin{lemma}\label{lem 2.6}
   Assume that {\rm (V1)-(V3)} and {\rm (F1)-(F3)} hold. Then $\Lambda\ne\emptyset$ and
 \begin{equation}\label{La1}
   \left\{u\in H^1(\R^N)\setminus \{0\} : \mathcal{P}^{\infty}(u)\le 0 \ \mbox{or} \ \mathcal{P}(u)\le 0\right\}\subset \Lambda.
 \end{equation}
 \end{lemma}

 \begin{proof}  In view of the proof of \cite[The proof of Claim 1 in Proposition 2.1]{MS}, (F3) implies $\Lambda\ne\emptyset$.
 Next, we have two cases to distinguish:

 \vskip2mm
 \par
  1). $u\in H^1(\R^N)\setminus \{0\}$ and $\mathcal{P}^{\infty}(u)\le 0$, then \eqref{Pi} implies $u\in \Lambda$.

 \par
  2). $u\in H^1(\R^N)\setminus \{0\}$ and $\mathcal{P}(u)\le 0$, then it follows from \eqref{Jv}, \eqref{B22} and \eqref{B28} that
 \begin{eqnarray*}
   &     &  -\frac{N+\alpha}{2}\int_{\R^N}(I_{\alpha}*F(u))F(u)\mathrm{d}x\\
   &  =  & \mathcal{P}(u)-\frac{N-2}{2}\|\nabla u\|_2^2-\frac{1}{2}\int_{{\R}^N}\left[NV(x)
             +\nabla V(x)\cdot x\right]u^2\mathrm{d}x\\
   & \le & -\frac{N-2}{2}\|\nabla u\|_2^2+\frac{(N-2)^3\theta}{8}\int_{{\R}^N}\frac{u^2}{|x|^2}\mathrm{d}x\\
   & \le & -\frac{(1-\theta)(N-2)}{2}\|\nabla u\|_2^2 < 0,
 \end{eqnarray*}
 which implies $u\in \Lambda$.
 \end{proof}

 \begin{lemma}\label{lem 2.7}
 Assume that {\rm (V1)-(V3)}  and {\rm (F1)-(F3)}  hold. Then for any
 $u\in \Lambda$, there exists a unique $t_u>0$ such that $u_{t_u}\in \mathcal{M}$.
 \end{lemma}

 \begin{proof} Let $u\in \Lambda$ be fixed and define a function $\zeta(t):=\mathcal{I}(u_t)$
 on $(0, \infty)$. Clearly, by \eqref{Jv} and \eqref{B23}, we have
 \begin{eqnarray}\label{B41}
  \zeta'(t)=0
    &     & \Leftrightarrow \ \ \frac{N-2}{2}t^{N-2}\|\nabla u\|_2^2+\frac{t^N}{2}\int_{{\R}^N}[NV(tx)
              +\nabla V(tx)\cdot (tx)]u^2\mathrm{d}x\nonumber\\
    &     & \ \ \ \ \ \ \ \ \ \
              -\frac{(N+\alpha)t^{N+\alpha}}{2}\int_{\R^N}(I_{\alpha}*F(u))F(u)\mathrm{d}x=0\nonumber\\
    &     & \ \Leftrightarrow  \ \ \mathcal{P}(u_t)=0 \ \ \Leftrightarrow  \ \ u_t\in \mathcal{M}.
 \end{eqnarray}
 It is easy to verify, using (V1), (V2), (F1), \eqref{B23} and the definition of $\Lambda$, that $\lim_{t\to 0}\zeta(t)=0$, $\zeta(t)>0$
 for $t>0$ small and $\zeta(t)<0$ for $t$ large. Therefore $\max_{t\in (0, \infty)}\zeta(t)$ is achieved at $t_u>0$ so that
 $\zeta'(t_u)=0$ and $u_{t_u}\in \mathcal{M}$.

 \par
    Next we claim that $t_u$ is unique for any $u\in \Lambda$. In fact, for any given $u\in \Lambda$,
 let $t_1, t_2>0$ such that $u_{t_1}, u_{t_2} \in \mathcal{M}$. Then $\mathcal{P}\left(u_{t_1}\right)=\mathcal{P}\left(u_{t_2}\right)=0$.
 Jointly with \eqref{B21}, we have
 \begin{eqnarray}\label{BB41}
   \mathcal{I}\left(u_{t_1}\right)
   & \ge & \mathcal{I}\left(u_{t_2}\right)+\frac{t_1^{N+\alpha}-t_2^{N+\alpha}}{(N+\alpha)t_1^{N+\alpha}}\mathcal{P}\left(u_{t_1}\right)
             +\frac{(1-\theta)\mathfrak{g}(t_2/t_1)}{2(N+\alpha)}\|\nabla u_{t_1}\|_2^2\nonumber\\
   &  =  & \mathcal{I}\left(u_{t_2}\right)+\frac{(1-\theta)t_1^{N-2}\mathfrak{g}(t_2/t_1)}{2(N+\alpha)}\|\nabla u\|_2^2
 \end{eqnarray}
 and
 \begin{eqnarray}\label{B42}
   \mathcal{I}\left(u_{t_2}\right)
   & \ge & \mathcal{I}\left(u_{t_1}\right)+\frac{t_2^{N+\alpha}-t_1^{N+\alpha}}{(N+\alpha)t_2^{N+\alpha}}\mathcal{P}\left(u_{t_2}\right)
            +\frac{(1-\theta)\mathfrak{g}(t_1/t_2)}{2(N+\alpha)}\|\nabla u_{t_2}\|_2^2\nonumber\\
   &  =  & \mathcal{I}\left(u_{t_1}\right)+\frac{(1-\theta)t_2^{N-2}\mathfrak{g}(t_1/t_2)}{2(N+\alpha)}\|\nabla u\|_2^2.
 \end{eqnarray}
 \eqref{B10}, \eqref{BB41} and \eqref{B42} imply $t_1=t_2$. Therefore, $t_u> 0$ is unique for any $u\in \Lambda$.
 \end{proof}

 \begin{corollary}\label{cor2.8}
 Assume that {\rm (F1)-(F3)}  hold. Then for any
 $u\in \Lambda$, there exists a unique $t_u>0$ such that $u_{t_u}\in \mathcal{M}^{\infty}$.
 \end{corollary}

    From Corollary \ref{cor2.4}, Lemmas \ref{lem 2.6}, \ref{lem 2.7} and Corollary \ref{cor2.8}, we have $\mathcal{M}\ne \emptyset$,
 $\mathcal{M}^{\infty}\ne \emptyset$ and the following lemma.

 \begin{lemma}\label{lem 2.9}
   Assume that {\rm (V1)-(V3)}  and {\rm (F1)-(F3) }  hold.  Then
 $$
   \inf_{u\in \mathcal{M}}\mathcal{I}(u)
   :=m=\inf_{u\in \Lambda}\max_{t > 0}\mathcal{I}(u_t).
 $$
 \end{lemma}

 \par
  The following lemma is a known result which can be proved by a standard argument(see \cite{TC1}).

 \begin{lemma}\label{lem 2.10}
  Assume that {\rm (V1), (F1)}  and {\rm (F2) } hold. If $u_n\rightharpoonup \bar{u}$ in $H^1(\R^N)$, then
 \begin{equation}\label{F60}
   \mathcal{I}(u_n)=\mathcal{I}(\bar{u})+\mathcal{I}(u_n-\bar{u})+o(1)
 \end{equation}
 and
 \begin{equation}\label{F63}
   \mathcal{P}(u_n)=\mathcal{P}(\bar{u})+\mathcal{P}(u_n-\bar{u})+o(1).
 \end{equation}
 \end{lemma}

 \begin{lemma}\label{lem 2.11}
  Assume that {\rm (V1)-(V3)}  and {\rm (F1)-(F3)} hold. Then
 \begin{enumerate}[{\rm(i)}]
  \item there exists $\rho_0>0$ such that $\|u\|\ge \rho_0, \ \forall \ u\in \mathcal{M}$;
  \item $m=\inf_{u\in \mathcal{M}}\mathcal{I}(u)>0$.
 \end{enumerate}
 \end{lemma}

 \begin{proof} (i). Since $\mathcal{P}(u)=0$ for all $ u\in \mathcal{M}$, by \eqref{Ru}, \eqref{Jv}, \eqref{B26} and Sobolev
 embedding theorem, one has
 \begin{eqnarray}\label{G62}
   \frac{\gamma_1}{2}\|u\|^2
   & \le & \frac{N-2}{2}\|\nabla u\|_2^2+\frac{1}{2}\int_{{\R}^N}[NV(x)+\nabla V(x)\cdot x]u^2\mathrm{d}x \nonumber\\
   &  =  & \frac{N+\alpha}{2}\int_{\R^N}(I_{\alpha}*F(u))F(u)\mathrm{d}x \nonumber\\
   & \le & \|u\|^{2(N+\alpha)/N}+C_1\|u\|^{2(N+\alpha)/(N-2)},
 \end{eqnarray}
 which implies
 \begin{equation}\label{G63}
   \|u\|\ge \rho_0:=\min\left\{1, \left[\frac{\gamma_1}{2(1+C_1)}\right]^{N/2\alpha}\right\}, \ \ \ \ \forall \ u\in \mathcal{M}.
 \end{equation}

 \vskip2mm
 \par
   (ii). Let $\{u_n\}\subset \mathcal{M}$ be such that $\mathcal{I}(u_n)\rightarrow m$. There are two possible cases:

 \par
  1) $\inf_{n\in\N}\|\nabla u_n\|_2>0$ \ and 2) $\inf_{n\in\N}\|\nabla u_n\|_2=0$.

 \vskip2mm
 \par
   Case 1). $\inf_{n\in\N}\|\nabla u_n\|_2:=\varrho_0>0$. In this case, from \eqref{B21} with $t \rightarrow 0$, we have
 \begin{equation*}
   m+o(1)=\mathcal{I}(u_n)
    \ge \frac{(1-\theta)(2+\alpha)}{2(N+\alpha)}\|\nabla u_n\|_2^2\ge \frac{(1-\theta)(2+\alpha)}{2(N+\alpha)}\varrho_0^2.
 \end{equation*}

 \par
   Case 2). $\inf_{n\in\N}\|\nabla u_n\|_2=0$. In this case, by \eqref{G63}, passing to a subsequence, one has
 \begin{equation}\label{Va1}
   \|\nabla u_n\|_2\to 0, \ \ \ \ \|u_n\|_2\ge \frac{1}{2}\rho_0.
 \end{equation}
 By \eqref{Ru} and the Sobolev inequality, one has for all $u\in H^1(\R^N)$,
 \begin{eqnarray}\label{G64}
   &     & \int_{\R^N}(I_{\alpha}*F(u))F(u)\mathrm{d}x\nonumber\\
   & \le & C_2\left(\|u\|_2^{2(N+\alpha)/N}+\|u\|_{2^*}^{2(N+\alpha)/(N-2)}\right)\nonumber\\
   & \le & C_2\left(\|u\|_2^{2(N+\alpha)/N}+S^{-(N+\alpha)/(N-2)}\|\nabla u\|_{2}^{2(N+\alpha)/(N-2)}\right).
 \end{eqnarray}
 By (V1), there exists $R>0$ such that $V(x)\ge \frac{V_{\infty}}{2}$ for $|x|\ge R$. This implies
 \begin{equation}\label{G65}
   \int_{|tx|\ge R}V(tx)u^2\mathrm{d}x\ge \frac{V_{\infty}}{2}\int_{|tx|\ge R}u^2\mathrm{d}x, \ \ \ \ \forall\ t>0,\ u\in H^1(\R^N).
 \end{equation}
 Making use of the H\"older inequality and the Sobolev inequality, we get
 \begin{eqnarray}\label{G66}
   \int_{|tx|< R}u^2\mathrm{d}x
    & \le & \left(\frac{\omega_N R^N}{t^N}\right)^{(2^*-2)/2^*}\left(\int_{|tx|< R}u^{2^*}\mathrm{d}x\right)^{2/2^*}\nonumber\\
    & \le & \omega_N^{2/N} R^2t^{-2}S^{-1}\|\nabla u\|_2^2, \ \ \ \ \forall\ t>0,\ u\in H^1(\R^N).
 \end{eqnarray}
 Let
 \begin{equation}\label{G67}
  \delta_0=\min\left\{V_{\infty}, SR^{-2}\omega_N^{-2/N}\right\}
 \end{equation}
 and
 \begin{equation}\label{G68}
  t_n=\left(\frac{\delta_0}{4C_2}\right)^{1/\alpha}\|u_n\|_2^{-2/N}.
 \end{equation}
 Then \eqref{Va1} implies $\{t_n\}$ is bounded.
 Thus it follows from \eqref{B23}, \eqref{Imax}, \eqref{Va1}, \eqref{G64}, \eqref{G65}, \eqref{G66}, \eqref{G67} and \eqref{G68} that
 \begin{eqnarray}\label{G69}
   m+o(1)
   &  =  & \mathcal{I}(u_n)\ge \mathcal{I}\left((u_n)_{t_n}\right)\nonumber\\
   &  =  &  \frac{t_n^{N-2}}{2}\|\nabla u_n\|_2^2+\frac{t_n^N}{2}\int_{{\R}^N}V(t_nx)u_n^2\mathrm{d}x
             -\frac{t_n^{N+\alpha}}{2}\int_{\R^N}(I_{\alpha}*F(u_n))F(u_n)\mathrm{d}x\nonumber\\
   & \ge & \frac{S}{2R^2\omega_N^{2/N}}t_n^N\int_{|t_nx|< R}u_n^2\mathrm{d}x
              +\frac{1}{4}V_{\infty}t_n^N\int_{|t_nx|\ge R}u_n^2\mathrm{d}x\nonumber\\
   &     &  \ \  -\frac{1}{2}C_2t_n^{N+\alpha}\|u_n\|_2^{2(N+\alpha)/N}
                 -\frac{C_2}{2S^{(N+\alpha)/(N-2)}}t_n^{N+\alpha}\|\nabla u_n\|_{2}^{2(N+\alpha)/(N-2)}\nonumber\\
   & \ge & \frac{1}{4}\delta_0t_n^N\|u_n\|_2^2-\frac{1}{2}C_2t_n^{N+\alpha}\|u_n\|_2^{2(N+\alpha)/N}+o(1)\nonumber\\
   &  =  & \frac{1}{4}t_n^N\|u_n\|_2^2\left(\delta_0-2C_2t_n^{\alpha}\|u_n\|_2^{2\alpha/N}\right)+o(1)\nonumber\\
   &  =  & \frac{\delta_0}{8}\left(\frac{\delta_0}{4C_2}\right)^{N/\alpha}+o(1).
 \end{eqnarray}
 Cases 1) and 2) show that $m=\inf_{u\in \mathcal{M}}\mathcal{I}(u)>0$.
 \end{proof}

 \begin{lemma}\label{lem 3.1}
   Assume that {\rm (V1)-(V3)} and {\rm (F1)-(F3)}  hold. Then $m\le m^{\infty}$.
 \end{lemma}

 \begin{proof} Arguing indirectly, we assume that $m> m^{\infty}$. Let $\varepsilon:=m-m^{\infty}$.
 Then there exists $u_{\varepsilon}^{\infty}$ such that
 \begin{equation}\label{P11}
   u_{\varepsilon}^{\infty}\in \mathcal{M}^{\infty} \ \ \ \  \mbox{and}
     \ \ \ \ m^{\infty}+\frac{\varepsilon}{2}>\mathcal{I}^{\infty}(u_{\varepsilon}^{\infty}).
 \end{equation}
 In view of Lemmas \ref{lem 2.6} and \ref{lem 2.7}, there exists $t_{\varepsilon}>0$ such that $(u_{\varepsilon}^{\infty})_{t_{\varepsilon}}\in \mathcal{M}$. Thus,
 it follows from (V1), (V2), \eqref{IU}, \eqref{Ii}, \eqref{B25} and \eqref{P11} that
 $$
   m^{\infty}+\frac{\varepsilon}{2}>\mathcal{I}^{\infty}(u_{\varepsilon}^{\infty})
     \ge \mathcal{I}^{\infty}\left((u_{\varepsilon}^{\infty})_{t_{\varepsilon}}\right)
     \ge \mathcal{I}\left((u_{\varepsilon}^{\infty})_{t_{\varepsilon}}\right)\ge m.
 $$
 This contradiction shows the conclusion of Lemma \ref{lem 3.1} is true.
 \end{proof}

 \begin{lemma}\label{lem 3.2}
  Assume that {\rm (V1)-(V3)}  and  {\rm (F1)-(F3)}  hold. Then $m$ is achieved.
 \end{lemma}

 \begin{proof}  In view of Lemma \ref{lem 2.11}, we have $m>0$. Let $\{u_n\}\subset \mathcal{M}$ be such that
 $\mathcal{I}(u_n)\rightarrow m$. Since $\mathcal{P}(u_n)=0$, then it follows from \eqref{B21} with $t \rightarrow 0$ that
 \begin{equation}\label{P24}
   m+o(1)= \mathcal{I}(u_n)\ge \frac{(1-\theta)(2+\alpha)}{2(N+\alpha)}\|\nabla u_n\|_2^2.
 \end{equation}
 This shows that $\{\|\nabla u_n\|_2\}$ is bounded. Next, we prove that $\{\|u_n\|_2\}$ is also bounded.
 Arguing by contradiction, suppose that $\|u_n\|_2 \to \infty$. By \eqref{Ru} and the Sobolev inequality,
 one has
\begin{eqnarray}\label{P25}
  &     & \int_{{\R}^N}(I_{\alpha}*F(u))F(u)\mathrm{d}x\nonumber\\
  & \le & \frac{\delta_0}{4}\left(\frac{\delta_0}{16m}\right)^{\alpha/N}\|u\|_2^{2(N+\alpha)/N}
            +C_3\|u\|_{2^*}^{2(N+\alpha)/(N-2)}\nonumber\\
  & \le & \frac{\delta_0}{4}\left(\frac{\delta_0}{16m}\right)^{\alpha/N}\|u\|_2^{2(N+\alpha)/N}
            +C_3S^{-(N+\alpha)/(N-2)}\|\nabla u\|_{2}^{2(N+\alpha)/(N-2)},\nonumber\\
  &     & \ \ \ \ \ \ \ \ \ \ \ \ \ \ \ \ \ \ \ \ \ \ \ \ \ \ \ \ \ \ \ \ \ \ \ \ \forall \ u\in H^1(\R^N),
 \end{eqnarray}
 where $\delta_0$ is given by \eqref{G67}.
 Let
 \begin{equation}\label{P26}
  \hat{t}_n=\left(\frac{16m}{\delta_0}\right)^{1/N}\|u_n\|_2^{-2/N}.
 \end{equation}
 Then $\hat{t}_n\to 0$. Thus it follows from \eqref{B23}, \eqref{Imax}, \eqref{G65}, \eqref{G66}, \eqref{G67}, \eqref{P25} and \eqref{P26} that
 \begin{eqnarray}\label{P27}
   m+o(1)
    &  =  & \mathcal{I}(u_n) \ge \mathcal{I}\left((u_n)_{\hat{t}_n}\right)\nonumber\\
    &  =  & \frac{\hat{t}_n^{N-2}}{2}\|\nabla u_n\|_2^2+\frac{\hat{t}_n^N}{2}\int_{{\R}^N}V(\hat{t}_nx)u_n^2\mathrm{d}x
             -\frac{\hat{t}_n^{N+\alpha}}{2}\int_{\R^N}(I_{\alpha}*F(u_n))F(u_n)\mathrm{d}x\nonumber\\
   & \ge & \frac{S}{2R^2\omega_N^{2/N}}\hat{t}_n^N\int_{|\hat{t}_nx|< R}u_n^2\mathrm{d}x
              +\frac{1}{4}V_{\infty}\hat{t}_n^N\int_{|\hat{t}_nx|\ge R}u_n^2\mathrm{d}x\nonumber\\
   &     &  \ \ -\frac{1}{8}\delta_0\left(\frac{\delta_0}{16m}\right)^{\alpha/N}\hat{t}_n^{N+\alpha}\|u_n\|_2^{2(N+\alpha)/N}
                 -\frac{C_3}{2S^{(N+\alpha)/(N-2)}}\hat{t}_n^{N+\alpha}\|\nabla u_n\|_{2}^{2(N+\alpha)/(N-2)}\nonumber\\
   & \ge & \frac{1}{4}\delta_0\hat{t}_n^N\|u_n\|_2^2-\frac{\delta_0}{8}\left(\frac{\delta_0}{16m}\right)^{\alpha/N}
                \hat{t}_n^{N+\alpha}\|u_n\|_2^{2(N+\alpha)/N}+o(1)\nonumber\\
   &  =  & \frac{\delta_0}{4}\hat{t}_n^N\|u_n\|_2^2\left[1-\frac{1}{2}
              \left(\frac{\delta_0\hat{t}_n^N\|u_n\|_2^2}{16m}\right)^{\alpha/N}\right]+o(1)\nonumber\\
   &  =  & 2m+o(1).
 \end{eqnarray}
 This contradiction shows that $\{\|u_n\|_2\}$ is also bounded. Hence, $\{u_n\}$ is bounded in $H^1(\R^N)$.
 Passing to a subsequence, we have $u_n\rightharpoonup \bar{u}$ in $H^1(\R^N)$. Then $u_n\rightarrow \bar{u}$
 in $L_{\mathrm{loc}}^s(\R^N)$ for $2\le s<2^*$ and $u_n\rightarrow \bar{u}$ a.e. in $\R^N$. There are two
 possible cases: i). $\bar{u}=0$ and ii). $\bar{u}\ne 0$.

 \vskip2mm
 \par
   Case i). $\bar{u}=0$, i.e. $u_n\rightharpoonup 0$ in $H^1(\R^N)$.
 Then $u_n\rightarrow 0$ in $L_{\mathrm{loc}}^s(\R^N)$ for $2\le s<2^*$ and $u_n\rightarrow 0$ a.e. in $\R^N$.
 By (V1) and (V3), it is easy to show that
 \begin{equation}\label{P32}
   \lim_{n\to\infty}\int_{\R^N}[V_{\infty}-V(x)]u_n^2\mathrm{d}x=
   \lim_{n\to\infty}\int_{\R^N}\nabla V(x)\cdot xu_n^2\mathrm{d}x=0.
 \end{equation}
 From \eqref{IU}, \eqref{Ii}, \eqref{Pi}, \eqref{Jv} and \eqref{P32}, one can get
 \begin{equation}\label{P37}
   \mathcal{I}^{\infty}(u_n)\rightarrow m, \ \ \ \ \mathcal{P}^{\infty}(u_n)\rightarrow 0.
 \end{equation}
 From Lemma \ref{lem 2.11} (i), \eqref{Pi} and \eqref{P37}, one has
 \begin{eqnarray}\label{P64}
   \min\{N-2, NV_{\infty}\}\rho_0^2
    & \le & \min\{N-2, NV_{\infty}\}\|u_n\|^2\nonumber\\
    & \le & (N-2)\|\nabla u_n\|_2^2+NV_{\infty}\|u_n\|_2^2\nonumber\\
    &  =  & (N+\alpha)\int_{\R^N}(I_{\alpha}*F(u_n))F(u_n)\mathrm{d}x+o(1).
 \end{eqnarray}
 Using \eqref{Ru}, \eqref{P64} and Lions' concentration compactness principle \cite[Lemma 1.21]{WM}, we can prove that
 there exist $\delta>0$ and a sequence $\{y_n\}\subset \R^N$ such that $\int_{B_1(y_n)}|u_n|^2\mathrm{d}x> \delta$. Let
 $\hat{u}_n(x)=u_n(x+y_n)$. Then we have $\|\hat{u}_n\|=\|u_n\|$ and
 \begin{equation}\label{P65}
   \mathcal{I}^{\infty}(\hat{u}_n)\rightarrow m, \ \ \ \ \mathcal{P}^{\infty}(\hat{u}_n)= o(1)\rightarrow 0,
     \ \ \ \ \int_{B_1(0)}|\hat{u}_n|^2\mathrm{d}x> \delta.
 \end{equation}
 Therefore, there exists $\hat{u}\in H^1(\R^N)\setminus \{0\}$ such that, passing to a subsequence,
 \begin{equation}\label{P72}
 \left\{
   \begin{array}{ll}
     \hat{u}_n\rightharpoonup \hat{u}, & \mbox{in} \ H^1(\R^N); \\
     \hat{u}_n\rightarrow \hat{u}, & \mbox{in} \ L_{\mathrm{loc}}^s(\R^N), \ \forall \ s\in [1, 2^*);\\
     \hat{u}_n\rightarrow \hat{u}, & \mbox{a.e. on} \ \R^N.
   \end{array}
 \right.
 \end{equation}
 Let $w_n=\hat{u}_n-\hat{u}$. Then \eqref{P72} and Lemma \ref{lem 2.10} yield
 \begin{equation}\label{D73}
    \mathcal{I}^{\infty}(\hat{u}_n)=\mathcal{I}^{\infty}(\hat{u})+\mathcal{I}^{\infty}(w_n)+o(1)
 \end{equation}
 and
 \begin{equation}\label{D74}
    \mathcal{P}^{\infty}(\hat{u}_n) = \mathcal{P}^{\infty}(\hat{u})+\mathcal{P}^{\infty}(w_n)+o(1).
 \end{equation}
 Set
 \begin{equation}\label{Ps0}
   \Psi_0(u)=\frac{(2+\alpha)\|\nabla u\|_2^2+\alpha V_{\infty}\|u\|_2^2}{2(N+\alpha)}.
 \end{equation}
 From \eqref{Ii}, \eqref{Pi}, \eqref{P65}, \eqref{D73} and \eqref{D74}, one has
 \begin{equation}\label{P75}
    \Psi_0(w_n)=m-\Psi_0(\hat{u})+o(1),
      \ \ \ \ \mathcal{P}^{\infty}(w_n) = -\mathcal{P}^{\infty}(\hat{u})+o(1).
 \end{equation}
 If there exists a subsequence $\{w_{n_i}\}$ of $\{w_n\}$ such that $w_{n_i}=0$, then going to this subsequence, we have
 \begin{equation}\label{P76}
    \mathcal{I}^{\infty}(\hat{u})=m, \ \ \ \ \mathcal{P}^{\infty}(\hat{u})=0.
 \end{equation}
 Next, we assume that $w_n\ne 0$. We claim that $\mathcal{P}^{\infty}(\hat{u})\le 0$. Otherwise, if $\mathcal{P}^{\infty}(\hat{u})>0$,
 then \eqref{P75} implies $\mathcal{P}^{\infty}(w_n) < 0$ for large $n$. In view of Lemma \ref{lem 2.6} and Corollary \ref{cor2.8}, there exists $t_n>0$ such that $(w_n)_{t_n}\in \mathcal{M}^{\infty}$ for large $n$. From \eqref{Ii}, \eqref{Pi}, \eqref{B25}, \eqref{Ps0} and \eqref{P75}, we obtain
 \begin{eqnarray*}
   m-\Psi_0(\hat{u})+o(1)
    &  =  & \Psi_0(w_n) = \mathcal{I}^{\infty}(w_n)-\frac{1}{N+\alpha}\mathcal{P}^{\infty}(w_n)\nonumber\\
    & \ge & \mathcal{I}^{\infty}\left({(w_n)}_{t_n}\right)-\frac{t_n^N}{N+\alpha}\mathcal{P}^{\infty}(w_n)\nonumber\\
    & \ge & m^{\infty}-\frac{t_n^N}{N+\alpha}\mathcal{P}^{\infty}(w_n)\ge m^{\infty},
 \end{eqnarray*}
 which implies $\mathcal{P}^{\infty}(\hat{u})\le 0$ due to $m\le m^{\infty}$ and $\Psi_0(\hat{u})>0$. Since $\hat{u}\ne 0$ and
 $\mathcal{P}^{\infty}(\hat{u})\le 0$, in view of Lemma \ref{lem 2.6} and Corollary \ref{cor2.8}, there exists
 $\hat{t}>0$ such that $\hat{u}_{\hat{t}}\in \mathcal{M}^{\infty}$.  From \eqref{Ii}, \eqref{Pi}, \eqref{B25}, \eqref{Ps0}, \eqref{P65}
 and the weak semicontinuity of norm, one has
 \begin{eqnarray*}
   m &  =  & \lim_{n\to\infty} \left[\mathcal{I}^{\infty}(\hat{u}_n)-\frac{1}{N+\alpha}\mathcal{P}^{\infty}(\hat{u}_n)\right]\nonumber\\
     &  =  & \lim_{n\to\infty}\Psi_0(\hat{u}_n)\ge \Psi_0(\hat{u})\nonumber\\
     &  =  & \mathcal{I}^{\infty}(\hat{u})-\frac{1}{N+\alpha}\mathcal{P}^{\infty}(\hat{u})\ge \mathcal{I}^{\infty}\left({\hat{u}}_{\hat{t}}\right)
               -\frac{\hat{t}^N}{N+\alpha}\mathcal{P}^{\infty}(\hat{u})\nonumber\\
     & \ge & m^{\infty}-\frac{\hat{t}^N}{N+\alpha}\mathcal{P}^{\infty}(\hat{u})\nonumber\\
     & \ge & m-\frac{\hat{t}^N}{N+\alpha}\mathcal{P}^{\infty}(\hat{u})\ge m,
 \end{eqnarray*}
 which implies \eqref{P76} holds also. In view of Lemmas \ref{lem 2.6} and \ref{lem 2.7}, there exists $\tilde{t}>0$ such that
 $\hat{u}_{\tilde{t}}\in \mathcal{M}$, moreover, it follows from (V1), (V2), \eqref{IU}, \eqref{Ii}, \eqref{P76} and Corollary \ref{cor2.3} that
 $$
   m\le \mathcal{I}(\hat{u}_{\tilde{t}})\le \mathcal{I}^{\infty}(\hat{u}_{\tilde{t}})\le \mathcal{I}^{\infty}(\hat{u})=m.
 $$
 This shows that $m$ is achieved at $\hat{u}_{\tilde{t}}\in \mathcal{M}$.

 \vskip2mm
 \par
   Case ii). $\bar{u}\ne 0$. Let $v_n=u_n-\bar{u}$. Then Lemma \ref{lem 2.10} yields
 \begin{equation}\label{K71}
    \mathcal{I}(u_n)=\mathcal{I}(\bar{u})+\mathcal{I}(v_n)+o(1)
 \end{equation}
 and
 \begin{equation}\label{K72}
   \mathcal{P}(u_n)=\mathcal{P}(\bar{u})+\mathcal{P}(v_n)+o(1).
 \end{equation}
 Set
 \begin{equation}\label{K73}
   \Psi(u)=\frac{2+\alpha}{2(N+\alpha)}\|\nabla u\|_2^2+\frac{1}{2(N+\alpha)}\int_{{\R}^N}[\alpha V(x)-(\nabla V(x), x)]u^2\mathrm{d}x.
 \end{equation}
 Then it follows from \eqref{B22} and \eqref{B27} that
 \begin{eqnarray}\label{K74}
   &     & (2+\alpha)\|\nabla u\|_2^2+\int_{{\R}^N}[\alpha V(x)-(\nabla V(x), x)]u^2\mathrm{d}x\nonumber\\
   & \ge & (2+\alpha)\|\nabla u\|_2^2-\frac{(2+\alpha)(N-2)^2\theta}{4}\int_{{\R}^N}\frac{u^2}{|x|^2}\mathrm{d}x\nonumber\\
   & \ge & (1-\theta)(2+\alpha)\|\nabla u\|_2^2, \ \ \ \ \forall \ u\in H^1(\R^N).
 \end{eqnarray}
 Since $\mathcal{I}(u_n)\rightarrow m$ and $\mathcal{P}(u_n)=0$, then it follows from \eqref{IU}, \eqref{Jv}, \eqref{K71}, \eqref{K72} and \eqref{K73} that
 \begin{equation}\label{K75}
    \Psi(v_n)=m-\Psi(\bar{u})+o(1), \ \ \ \ \mathcal{P}(v_n) = -\mathcal{P}(\bar{u})+o(1).
 \end{equation}
 If there exists a subsequence $\{v_{n_i}\}$ of $\{v_n\}$ such that $v_{n_i}=0$, then going to this subsequence, we have
 \begin{equation}\label{K76}
    \mathcal{I}(\bar{u})=m, \ \ \ \ \mathcal{P}(\bar{u})=0,
 \end{equation}
 which implies the conclusion of Lemma \ref{lem 3.2} holds. Next, we assume that $v_n\ne 0$. We claim that $\mathcal{P}(\bar{u})\le 0$.
 Otherwise $\mathcal{P}(\bar{u})>0$, then \eqref{K75} implies $\mathcal{P}(v_n) < 0$ for large $n$. In view of Lemmas \ref{lem 2.6}
 and \ref{lem 2.7}, there exists $t_n>0$ such that $(v_n)_{t_n}\in \mathcal{M}$ for large $n$.  From \eqref{IU}, \eqref{Jv}, \eqref{B21}, \eqref{K73} and \eqref{K75}, we obtain
 \begin{eqnarray*}
   m-\Psi(\bar{u})+o(1)
    &  =  & \Psi(v_n) = \mathcal{I}(v_n)-\frac{1}{N+\alpha}\mathcal{P}(v_n)\nonumber\\
    & \ge & \mathcal{I}\left({(v_n)}_{t_n}\right)-\frac{t_n^N}{N+\alpha}\mathcal{P}(v_n)\nonumber\\
    & \ge & m-\frac{t_n^N}{N+\alpha}\mathcal{P}(v_n)\ge m,
 \end{eqnarray*}
 which implies $\mathcal{P}(\bar{u})\le 0$ due to $\Psi(\bar{u})>0$. Since $\bar{u}\ne 0$ and $\mathcal{P}(\bar{u})\le 0$,
 in view of Lemmas \ref{lem 2.6} and \ref{lem 2.7}, there exists $\bar{t}>0$ such that $\bar{u}_{\bar{t}}\in \mathcal{M}$.
 From \eqref{IU}, \eqref{Jv}, \eqref{B21}, \eqref{K73}, \eqref{K74} and the weak semicontinuity of norm, one has
 \begin{eqnarray*}
   m
   &  =  & \lim_{n\to\infty} \left[\mathcal{I}(u_n)-\frac{1}{N+\alpha}\mathcal{P}(u_n)\right]
           = \lim_{n\to\infty} \Psi(u_n)\ge \Psi(\bar{u})\nonumber\\
   &  =  & \mathcal{I}(\bar{u})-\frac{1}{N+\alpha}\mathcal{P}(\bar{u})\ge \mathcal{I}\left({\bar{u}}_{\bar{t}}\right)
             -\frac{\bar{t}^N}{N+\alpha}\mathcal{P}(\bar{u})\nonumber\\
   & \ge & m-\frac{\bar{t}^N}{N+\alpha}\mathcal{P}(\bar{u})\ge m,
 \end{eqnarray*}
 which implies \eqref{K76} also holds.
 \end{proof}

 \begin{lemma}\label{lem 2.13}
   Assume that {\rm (V1)-(V3)}  and {\rm (F1)-(F3)} hold. If $\bar{u}\in \mathcal{M}$
 and $\mathcal{I}(\bar{u})=m$, then $\bar{u}$ is a critical point of $\mathcal{I}$.
 \end{lemma}

 \begin{proof}  Similar to the proof of \cite[Lemma 2.12]{TC3}, we can conclude above conclusion by using
 \begin{eqnarray}\label{B90}
   \mathcal{I}\left(\bar{u}_t\right)  \le \mathcal{I}(\bar{u})-\frac{(1-\theta)\mathfrak{g}(t)}{2(N+\alpha)}\|\nabla \bar{u}\|_2^2
    = m-\frac{(1-\theta)\mathfrak{g}(t)}{2(N+\alpha)}\|\nabla \bar{u}\|_2^2, \ \ \ \ \forall \ t> 0
 \end{eqnarray}
 and
 $$
   \varepsilon:=\min\left\{\frac{(1-\theta)\mathfrak{g}(0.5)}{5(N+\alpha)}\|\nabla \bar{u}\|_2^2,
   \frac{(1-\theta)\mathfrak{g}(1.5)}{5(N+\alpha)}\|\nabla \bar{u}\|_2^2, 1, \frac{\varrho\delta}{8}\right\}.
 $$
 instead of \cite[(2.35) and $\varepsilon$] {TC3}, respectively.
 \end{proof}

 \begin{proof} [Proof of Theorem \ref{thm1.1}]\ \ In view of Lemmas \ref{lem 2.9}, \ref{lem 3.2} and \ref{lem 2.13}, there exists $\bar{u}\in \mathcal{M}$ such that
 $$
   \mathcal{I}(\bar{u})=m=\inf_{u\in \Lambda}\max_{t > 0}\mathcal{I}(u_t)>0, \ \ \ \ \mathcal{I}'(\bar{u})=0.
 $$
 This shows that $\bar{u}$ is a nontrivial solution of \eqref{SE}.
 \end{proof}

 \begin{proof}[Proof of Theorem  \ref{thm1.3}]\ \ Let
 $$
   \mathcal{K}^{\infty}:=\left\{u\in H^1(\R^N)\setminus \{0\} : (\mathcal{I}^{\infty})'(u)=0\right\}, \ \ \ \
     \hat{m}^{\infty}:=\inf_{u\in\mathcal{K}^{\infty}}\mathcal{I}^{\infty}(u).
 $$
 On the one hand, in view of Corollary \ref{cor1.2}, there exists $\bar{u}\in \mathcal{M}^{\infty}$ such that $\mathcal{I}^{\infty}(\bar{u})=m^{\infty}$
 and $ (\mathcal{I}^{\infty})'(\bar{u})=0$. This shows that $\mathcal{K}^{\infty}\ne \emptyset$ and $\hat{m}^{\infty}\le m^{\infty}$.
 On the other hand, if $w\in \mathcal{K}^{\infty}$, then it follows from \eqref{Pi} (i.e. \cite[Theorem 3]{MS}) that $w\in \mathcal{M}^{\infty}$. Thus, $\mathcal{I}^{\infty}(w)\ge m^{\infty}$ for all $w\in \mathcal{K}^{\infty}$, which yields that $\hat{m}^{\infty}\ge m^{\infty}$. Therefore, $\hat{m}^{\infty}= m^{\infty}$.
 \end{proof}

 \vskip6mm
 {\section{The least energy solutions for \eqref{SE}}}
 \setcounter{equation}{0}

 \vskip2mm
 \par
   In this section, we give the proof of Theorem \ref{thm1.6}.

 \begin{proposition}\label{pro 4.1}{\rm\cite{JTo}}
 Let $X$ be a Banach space and let $J\subset \R^+$ be an interval, and
 $$
   \Phi_{\lambda}(u)=A(u)-\lambda B(u), \ \ \ \ \forall \ \lambda\in J,
 $$
 be a family of $\mathcal{C}^1$-functional on $X$ such that
 \begin{enumerate}[{\rm i)}]
  \item either $A(u)\to +\infty$ or $B(u)\to +\infty$, as $\|u\| \to \infty$;
  \item $B$ maps every bounded set of $X$ into a set of $\R$ bounded below;
  \item there are two points $v_1, v_2$ in $X$ such that
 \begin{equation}\label{cm}
   \tilde{c}_{\lambda}:=\inf_{\gamma\in \tilde{\Gamma}}\max_{t\in [0, 1]}\Phi_{\lambda}(\gamma(t))>\max\{\Phi_{\lambda}(v_1), \Phi_{\lambda}(v_2)\},
 \end{equation}
 \end{enumerate}
 where
 $$
   \tilde{\Gamma}=\left\{\gamma\in \mathcal{C}([0, 1], X): \gamma(0)=v_1, \gamma(1)=v_2\right\}.
 $$
 Then, for almost every $\lambda\in J$, there exists a sequence $\{u_n(\lambda)\}$ such that
 \begin{enumerate}[{\rm i)}]
  \item $\{u_n(\lambda)\}$ is bounded in $X$;
  \item $\Phi_{\lambda}(u_n(\lambda))\rightarrow c_{\lambda}$;
  \item $\Phi_{\lambda}'(u_n(\lambda))\rightarrow 0$ in $X^*$, where $X^*$ is the dual of $X$.
 \end{enumerate}
\end{proposition}

 \vskip4mm
 \par
    Similar to the proof of \cite[Theorem 3]{MS}, we can prove the following lemma.

\begin{lemma}\label{lem 4.2}
 Assume that {\rm(V1), (F1) and (F2)} hold. Let $u$ be a critical point of
 $\mathcal{I}_{\lambda}$ in $H^1(\R^N)$, then we have the following Poho\u zaev type identity
 \begin{eqnarray}\label{Pl}
  \mathcal{P}_{\lambda}(u)
    & :=  & \frac{N-2}{2}\|\nabla u\|_2^2+\frac{1}{2}\int_{\R^N}\left[NV(x)+\nabla V(x)\cdot x\right]u^2\mathrm{d}x\nonumber\\
    &     & \ \  -\frac{(N+\alpha)\lambda}{2}\int_{\R^N}(I_{\alpha}*F(u))F(u)\mathrm{d}x=0.
 \end{eqnarray}
 \end{lemma}

 \par
    By Corollary \ref{cor2.3}, we have the following lemma.

 \begin{lemma}\label{lem 4.3}
 Assume that {\rm(F1)} and {\rm(F2)}  hold. Then
 \begin{eqnarray}\label{G45}
   \mathcal{I}_{\lambda}^{\infty}(u)
    &  =  & \mathcal{I}_{\lambda}^{\infty}\left(u_t\right)+\frac{1-t^N}{N}\mathcal{P}_{\lambda}^{\infty}(u)
              +\frac{\mathfrak{g}(t)\|\nabla u\|_2^2+V_{\infty}\mathfrak{h}(t)\|u\|_2^2}{2(N+\alpha)},\nonumber\\
    &     &  \ \ \ \ \ \ \ \ \forall \ u\in H^1(\R^N), \ \ t > 0, \ \ \lambda\ge 0.
 \end{eqnarray}
 \end{lemma}

 \par
    In view of Corollary \ref{cor1.2}, $\mathcal{I}_1^{\infty}=\mathcal{I}^{\infty}$ has a minimizer $u_1^{\infty}\ne 0$ on $\mathcal{M}_1^{\infty}
 =\mathcal{M}^{\infty}$,  i.e.
 \begin{equation}\label{G47}
   u_1^{\infty}\in \mathcal{M}_1^{\infty}, \ \ \ \ (\mathcal{I}_1^{\infty})'(u_1^{\infty})=0
     \ \ \ \  \mbox{and} \ \ \ \ m_1^{\infty}=\mathcal{I}_1^{\infty}(u_1^{\infty}),
 \end{equation}
 where $m_{\lambda}^{\infty}$ is defined by \eqref{cm0}. Since \eqref{SE1} is autonomous, $V\in \mathcal{C}(\R^N, \R)$ and
 $V(x)\le V_{\infty}$ but $V(x)\not\equiv V_{\infty}$, then there exist $\bar{x}\in \R^N$ and $\bar{r}>0$ such that
 \begin{equation}\label{G48}
    V_{\infty}-V(x)>0, \ \ |u_1^{\infty}(x)|>0\ \ \ \ a.e. \ |x-\bar{x}|\le \bar{r}.
 \end{equation}

 \par
   By (V1), we have $V_{\max}:=\max_{x\in\R^N}V(x)\in (0,\infty)$. Let
 \begin{equation}\label{Il*}
   \mathcal{I}_{\lambda}^{*}(u)=\frac{1}{2}\int_{\R^N}\left(|\nabla u|^2+V_{\max}u^2\right)\mathrm{d}x
       -\frac{\lambda}{2}\int_{\R^N}(I_{\alpha}*F(u))F(u)\mathrm{d}x.
 \end{equation}
 Then it follows from \eqref{B23} and \eqref{G47} that there exists $T>0$ such that
 \begin{equation}\label{DT}
   I^*_{1/2}\left((u_1^{\infty})_{t}\right)<0, \ \ \ \ \forall \ t\ge T.
 \end{equation}

 \begin{lemma}\label{lem 4.4}
 Assume that {\rm(V1)} and {\rm(F1)-(F3)} hold. Then
 \begin{enumerate}[{\rm(i)}]
 \par
 \item $\mathcal{I}_{\lambda}\left((u_1^{\infty})_{T}\right)<0$ for all $\lambda\in [0.5, 1]$;

 \item there exists a positive constant $\kappa_0 $ independent of $\lambda$ such that for all $\lambda\in [0.5, 1]$,
 \begin{equation*}
   c_{\lambda}:=\inf_{\gamma\in \Gamma}\max_{t\in [0, 1]}\mathcal{I}_{\lambda}(\gamma(t))\ge \kappa_0
     >\max\left\{\mathcal{I}_{\lambda}(0), \mathcal{I}_{\lambda}\left((u_1^{\infty})_{T}\right)\right\},
 \end{equation*}
 where
 \begin{equation}\label{Ga}
   \Gamma=\left\{\gamma\in \mathcal{C}([0, 1], H^1(\R^N)): \gamma(0)=0, \gamma(1)=(u_1^{\infty})_{T}\right\};
 \end{equation}

 \item $c_{\lambda}$ is bounded for $\lambda\in [0.5, 1]$;

 \item $m_{\lambda}^{\infty}$ is non-increasing on $\lambda\in [0.5, 1]$;

 \item $\limsup_{\lambda\to \lambda_0}c_{\lambda}\le c_{\lambda_0}$ for $\lambda_0\in(0.5, 1]$.
 \end{enumerate}
 \end{lemma}

 \par
   Since $m_{\lambda}^{\infty}=\mathcal{I}_{\lambda}^{\infty}(u_{\lambda}^{\infty})$ and $\int_{\R^N}(I_{\alpha}*F(u_{\lambda}^{\infty}))F(u_{\lambda}^{\infty})\mathrm{d}x>0$, then the proof of (i)-(iv) in Lemma \ref{lem 4.4} is
 standard, (v) can be proved similar to \cite[Lemma 2.3]{Je}, so we omit it.

 \begin{lemma}\label{lem 4.5}
 Assume that {\rm(V1), (V2)} and {\rm(F1)-(F3)} hold. Then
 there exists $\bar{\lambda}\in [1/2, 1)$ such that $c_{\lambda}<m_{\lambda}^{\infty}$ for $\lambda\in (\bar{\lambda}, 1]$.
 \end{lemma}

 \begin{proof} It is easy to see that $\mathcal{I}_{\lambda}\left((u_1^{\infty})_t\right)$ is continuous on $t\in (0, \infty)$.
 Hence for any $\lambda\in [1/2, 1]$,  we can choose $t_{\lambda}\in (0, T)$ such that $\mathcal{I}_{\lambda}
 \left((u_1^{\infty})_{t_{\lambda}}\right) =\max_{t\in [0,T]}\mathcal{I}_{\lambda}\left((u_1^{\infty})_t\right)$. Setting
 \begin{equation}\label{ga0}
   \gamma_0(t)=\left\{\begin{array}{ll}
    (u_1^{\infty})_{(tT)}, \ \ &\mbox{for} \ t>0,\\
    0, \ \ & \mbox{for} \ t=0.
    \end{array}\right.
 \end{equation}
 Then $\gamma_0\in \Gamma$ defined by Lemma \ref{lem 4.4} (ii). Moreover
 \begin{equation}\label{G50}
   \mathcal{I}_{\lambda} \left((u_1^{\infty})_{t_{\lambda}}\right)=\max_{t\in [0,1]}\mathcal{I}_{\lambda}\left(\gamma_0(t)\right)
       \ge c_{\lambda}.
 \end{equation}
 Since $\mathcal{P}^{\infty}(u_1^{\infty})=0$, then $\int_{\R^N}(I_{\alpha}*F(u_1^{\infty}))F(u_1^{\infty})\mathrm{d}x>0$.
 Let
 \begin{equation}\label{G56}
   \zeta_0:=\min\{3\bar{r}/8(1+|\bar{x}|), 1/4\}.
 \end{equation}
 Then it follows from \eqref{G48} and \eqref{G56} that
 \begin{equation}\label{G58}
   |x-\bar{x}|\le \frac{\bar{r}}{2} \ \ \mbox{and} \ \ s\in [1-\zeta_0, 1+\zeta_0] \Rightarrow |sx-\bar{x}|\le \bar{r}.
 \end{equation}
 Let
 \begin{eqnarray}\label{G59}
   \bar{\lambda}
    & :=  & \max\left\{\frac{1}{2}, 1-\frac{(1-\zeta_0)^N\min_{s\in [1-\zeta_0, 1+\zeta_0]}\int_{\R^N}\left[V_{\infty}-
              V(sx)\right]|u_1^{\infty}|^2\mathrm{d}x}{T^{N+\alpha}\int_{\R^N}(I_{\alpha}*F(u_1^{\infty}))
               F(u_1^{\infty})\mathrm{d}x},\right.\nonumber\\
    &     & \left.1-\frac{\min\{\mathfrak{g}(1-\zeta_0),\mathfrak{g}(1+\zeta_0)\}\|\nabla u_1^{\infty}\|_2^2
             +V_{\infty}\min\{\mathfrak{h}(1-\zeta_0),\mathfrak{h}(1+\zeta_0)\}\|u_1^{\infty}\|_2^2}{(N+\alpha)
             T^{N+\alpha}\int_{\R^N}(I_{\alpha}*F(u_1^{\infty}))F(u_1^{\infty})\mathrm{d}x}\right\}. \ \ \
 \end{eqnarray}
 Then it follows from \eqref{B10}, \eqref{B20}, \eqref{G48} and \eqref{G58} that $1/2\le \bar{\lambda}<1$. We have two cases to distinguish:

 \vskip2mm
 \par
   Case i). $t_{\lambda}\in [1-\zeta_0, 1+\zeta_0]$. From \eqref{Ilu}, \eqref{Il}, \eqref{G45}-\eqref{G50}, \eqref{G58}, \eqref{G59} and
 Lemma \ref{lem 4.4} (iv), we have
 \begin{eqnarray*}
   m_{\lambda}^{\infty}
    & \ge & m_1^{\infty}=\mathcal{I}_1^{\infty}(u_1^{\infty})\ge \mathcal{I}_1^{\infty}\left((u_1^{\infty})_{t_{\lambda}}\right)\nonumber\\
    &  =  & \mathcal{I}_{\lambda}\left((u_1^{\infty})_{t_{\lambda}}\right)
              -\frac{(1-\lambda)t_{\lambda}^{N+\alpha}}{2}\int_{\R^N}(I_{\alpha}*F(u_1^{\infty}))F(u_1^{\infty})\mathrm{d}x\nonumber\\
    &     & \ \   +\frac{t_{\lambda}^N}{2}\int_{\R^N}[V_{\infty}-V(t_{\lambda}x)]|u_1^{\infty}|^2\mathrm{d}x\nonumber\\
    & \ge & c_{\lambda} -\frac{(1-\lambda)T^{N+\alpha}}{2}\int_{\R^N}(I_{\alpha}*F(u_1^{\infty}))F(u_1^{\infty})\mathrm{d}x\nonumber\\
    &     & \ \     +\frac{(1-\zeta_0)^N}{2}\min_{s\in [1-\zeta_0, 1+\zeta_0]}
                \int_{\R^N}\left[V_{\infty}-V(sx)\right]|u_1^{\infty}|^2\mathrm{d}x\nonumber\\
    &  >  & c_{\lambda}, \ \ \ \ \forall \ \lambda\in (\bar{\lambda}, 1].
 \end{eqnarray*}
 \par
   Case ii). $t_{\lambda}\in (0, 1-\zeta_0)\cup (1+\zeta_0, T]$. From \eqref{Ilu}, \eqref{Il}, \eqref{B10}, \eqref{B20}, \eqref{G45}, \eqref{G47}, \eqref{G50}, \eqref{G59} and Lemma \ref{lem 4.4} (iv), we have
 \begin{eqnarray*}
   m_{\lambda}^{\infty}
    & \ge & m_1^{\infty}=\mathcal{I}_1^{\infty}(u_1^{\infty})= \mathcal{I}_1^{\infty}\left((u_1^{\infty})_{t_{\lambda}}\right)
             +\frac{\mathfrak{g}(t_{\lambda})\|\nabla u_1^{\infty}\|_2^2
             +V_{\infty}\mathfrak{h}(t_{\lambda})\|u_1^{\infty}\|_2^2}{2(N+\alpha)}\nonumber\\
    &  =  & \mathcal{I}_{\lambda}\left((u_1^{\infty})_{t_{\lambda}}\right)
              -\frac{(1-\lambda)t_{\lambda}^{N+\alpha}}{2}\int_{\R^N}(I_{\alpha}*F(u_1^{\infty}))F(u_1^{\infty})\mathrm{d}x\nonumber\\
    &     & \ \  +\frac{t_{\lambda}^N}{2}\int_{\R^N}[V_{\infty}-V(t_{\lambda}x)]|u_1^{\infty}|^2\mathrm{d}x
              +\frac{\mathfrak{g}(t_{\lambda})\|\nabla u_1^{\infty}\|_2^2
             +V_{\infty}\mathfrak{h}(t_{\lambda})\|u_1^{\infty}\|_2^2}{2(N+\alpha)}\nonumber\\
    & \ge & c_{\lambda} -\frac{(1-\lambda)T^{N+\alpha}}{2}\int_{\R^N}(I_{\alpha}*F(u_1^{\infty}))F(u_1^{\infty})\mathrm{d}x\nonumber\\
    &     & \ \  +\frac{\min\{\mathfrak{g}(1-\zeta_0),\mathfrak{g}(1+\zeta_0)\}\|\nabla u_1^{\infty}\|_2^2
             +V_{\infty}\min\{\mathfrak{h}(1-\zeta_0),\mathfrak{h}(1+\zeta_0)\}\|u_1^{\infty}\|_2^2}{2(N+\alpha)}\nonumber\\
    &  >  & c_{\lambda}, \ \ \ \ \forall \ \lambda\in (\bar{\lambda}, 1].
 \end{eqnarray*}
 In both cases, we obtain that $c_{\lambda}<m_{\lambda}^{\infty}$ for $\lambda\in (\bar{\lambda}, 1]$.
 \end{proof}

 \begin{lemma}\label{lem 4.6}{\rm\cite{JT2}}
 Assume that {\rm(V1)} and {\rm(F1)-(F3)} hold. Let $\{u_n\}$ be a bounded {\rm(PS)}-
 sequence for $\mathcal{I}_{\lambda}$, for $\lambda\in  [1/2, 1]$. Then there exists a subsequence of $\{u_n\}$, still denoted by
 $\{u_n\}$, an integer $l\in \N \cup \{0\}$, a sequence $\{y_n^k\}$ and $w^k\in H^1(\R^3)$ for $1\le k\le l$, such that
 \begin{enumerate}[{\rm(i)}]
  \item $u_n\rightharpoonup u_0$ with $\mathcal{I}_{\lambda}'(u_0)=0$;
  \item $w^k\ne 0$ and $(\mathcal{I}_{\lambda}^{\infty})'(w^k)=0$ for $1\le k\le l$;
  \item $\left\|u_n-u_0-\sum_{k=1}^lw^k(\cdot+y_n^k)\right\|\rightarrow 0$;
  \item $\mathcal{I}_{\lambda}(u_n)\rightarrow \mathcal{I}_{\lambda}(u_0)+\sum_{i=1}^{l}\mathcal{I}_{\lambda}^{\infty}(w^i)$;
 \end{enumerate}
 where we agree that in the case $l = 0$ the above holds without $w^k$.
 \end{lemma}

 \begin{lemma}\label{lem 4.7}
 Assume that {\rm(V1)} and {\rm(V4)} hold. Then there exists $\gamma_3>0$ such that
 \begin{equation}\label{X31}
   (2+\alpha)\|\nabla u\|_2^2+\int_{\R^N}\left[\alpha V(x)-\nabla V(x)\cdot x\right]u^2\mathrm{d}x\ge \gamma_3\|u\|^2,
    \ \ \ \ \forall \ u\in H^1(\R^N).
 \end{equation}
 \end{lemma}

 \begin{proof} From (V1), (V4) and \eqref{B22}, we have
 \begin{eqnarray*}
    &     & (2+\alpha)\|\nabla u\|_2^2+\int_{\R^N}\left[\alpha V(x)
             -\nabla V(x)\cdot x\right]u^2\mathrm{d}x\nonumber\\
    &  =  & (2+\alpha)\|\nabla u\|_2^2-\frac{(N-2)^2}{2}\int_{\R^N}\frac{u^2}{|x|^2}\mathrm{d}x\nonumber\\
    &     &  \ \    +\int_{\R^N}\left[\alpha V(x)-\nabla V(x)\cdot x+\frac{(N-2)^2}{2|x|^2}\right]u^2\mathrm{d}x\nonumber\\
    & \ge & \alpha\|\nabla u\|_2^2 +(1-\theta')\alpha \int_{\R^N}V(x)u^2\mathrm{d}x\nonumber\\
    & \ge & \gamma_3\|u\|^2
 \end{eqnarray*}
 for some $\gamma_3>0$ due to (V1).
 \end{proof}

 \begin{lemma}\label{lem 4.8}
 Assume that {\rm(V1), (V2), (V4)} and {\rm(F1)-(F3)} hold. Then for almost every
 $\lambda\in (\bar{\lambda},1]$, there exists $u_{\lambda}\in H^1(\R^N)\setminus \{0\}$ such that
 \begin{equation}\label{D31}
   \mathcal{I}_{\lambda}'(u_{\lambda})=0, \ \ \ \ \mathcal{I}_{\lambda}(u_{\lambda}) = c_{\lambda}.
 \end{equation}
 \end{lemma}

 \begin{proof} Under (V1), (V2) and (F1)-(F3), Lemma \ref{lem 4.4} implies that $\mathcal{I}_{\lambda}(u)$ satisfies the assumptions
 of Proposition \ref{pro 4.1} with $X=H^1(\R^N)$, $\Phi_{\lambda}=\mathcal{I}_{\lambda}$  and $J=(\bar{\lambda},1]$. So for almost every
 $\lambda\in (\bar{\lambda},1]$, there exists a bounded sequence $\{u_n(\lambda)\} \subset H^1(\R^N)$ (for simplicity, we denote
 the sequence by $\{u_n\}$ instead of $\{u_n(\lambda)\}$) such that
 \begin{equation}\label{PS}
   \mathcal{I}_{\lambda}(u_n)\rightarrow c_{\lambda}>0, \ \ \ \ \mathcal{I}_{\lambda}'(u_n) \rightarrow 0.
 \end{equation}
 By Lemmas \ref{lem 4.2} and \ref{lem 4.6}, there exist a subsequence of $\{u_n\}$, still denoted by $\{u_n\}$, $u_{\lambda}\in H^1(\R^N)$,
 an integer $l\in \N\cup \{0\}$, and $w^1, \ldots, w^l\in H^1(\R^N)\setminus \{0\}$ such that
 \begin{equation}\label{un1}
   u_n\rightharpoonup u_{\lambda}\ \ \mbox {in} \  H^1(\R^N), \ \ \ \  \mathcal{I}_{\lambda}'(u_{\lambda})=0,
 \end{equation}
 \begin{equation}\label{un2}
   (\mathcal{I}_{\lambda}^{\infty})'(w^k)=0, \ \ \ \ \mathcal{I}_{\lambda}^{\infty}(w^k)\ge m_{\lambda}^{\infty},\ \ \ \ 1\le k\le l
 \end{equation}
 and
 \begin{equation}\label{Ab1}
    c_{\lambda}= \mathcal{I}_{\lambda}(u_{\lambda})+\sum_{k=1}^{l}\mathcal{I}_{\lambda}^{\infty}(w^k).
 \end{equation}

 \par
   Since $\mathcal{I}_{\lambda}'(u_{\lambda})=0$, then it follows from Lemma \ref{lem 4.2} that
 \begin{eqnarray}\label{Plt}
  \mathcal{P}_{\lambda}(u_{\lambda})
    &  =  & \frac{N-2}{2}\|\nabla u_{\lambda}\|_2^2+\frac{1}{2}\int_{\R^N}\left[NV(x)+\nabla V(x)\cdot x\right]u_{\lambda}^2\mathrm{d}x\nonumber\\
    &     & \ \  -\frac{(N+\alpha)\lambda}{2}\int_{\R^N}(I_{\alpha}*F(u_{\lambda}))F(u_{\lambda})\mathrm{d}x=0.
 \end{eqnarray}
 Since $\|u_n\|\nrightarrow 0$, we deduce from \eqref{un2} and \eqref{Ab1} that if $u_{\lambda}=0$ then $l\ge 1$ and
 \begin{eqnarray*}
   c_{\lambda} =  \mathcal{I}_{\lambda}(u_{\lambda})+\sum_{k=1}^{l}\mathcal{I}_{\lambda}^{\infty}(w^k)
    \ge  m_{\lambda}^{\infty},
 \end{eqnarray*}
 which contradicts with Lemma \ref{lem 4.5}. Thus $u_{\lambda}\ne 0$.
 It follows from \eqref{Ilu}, \eqref{X31} and \eqref{Plt} that
 \begin{eqnarray}\label{D33}
   \mathcal{I}_{\lambda}(u_{\lambda})
     &  =  & \mathcal{I}_{\lambda}(u_{\lambda})-\frac{1}{N+\alpha}\mathcal{P}_{\lambda}(u_{\lambda})\nonumber\\
    &  =  & \frac{2+\alpha}{2(N+\alpha)}\|\nabla u_{\lambda}\|_2^2+\frac{1}{2(N+\alpha)}\int_{\R^N}\left[\alpha V(x)
             -\nabla V(x)\cdot x\right]u_{\lambda}^2\mathrm{d}x\nonumber\\
    & \ge & \frac{\gamma_3}{2(N+\alpha)}\|u_{\lambda}\|^2>0.
 \end{eqnarray}
 From \eqref{Ab1} and \eqref{D33}, one has
 \begin{eqnarray}\label{D34}
   c_{\lambda} =  \mathcal{I}_{\lambda}(u_{\lambda})+\sum_{k=1}^{l}\mathcal{I}_{\lambda}^{\infty}(w^k)
    >  lm_{\lambda}^{\infty}.
 \end{eqnarray}
 By Lemma \ref{lem 4.5}, we have $c_{\lambda}<m_{\lambda}^{\infty}$ for $\lambda\in (\bar{\lambda}, 1]$, which, together with
 \eqref{D34}, implies that $l=0$ and $\mathcal{I}_{\lambda}(u_{\lambda}) = c_{\lambda}$.
 \end{proof}

 \begin{lemma}\label{lem 4.9}
 Assume that {\rm(V1), (V2), (V4)} and {\rm(F1)-(F3)} hold. Then there exists
 $\bar{u}\in H^1(\R^N)\setminus \{0\}$ such that
 \begin{equation}\label{Q05}
   \mathcal{I}'(\bar{u})=0, \ \ \ \ 0<\mathcal{I}(\bar{u}) \le c_1.
 \end{equation}
 \end{lemma}

 \begin{proof} In view of Lemmas \ref{lem 4.4} (iii) and \ref{lem 4.8}, there exist two sequences $\{\lambda_n\}\subset (\bar{\lambda}, 1]$
 and $\{u_{\lambda_n}\}\subset H^1(\R^N)\setminus \{0\}$, denoted by $\{u_n\}$, such that
 \begin{equation}\label{Q00}
   \lambda_n\rightarrow 1, \ \ \ \ c_{\lambda_n}\rightarrow c_*,\ \ \ \ \mathcal{I}_{\lambda_n}'(u_n)=0,
   \ \ \ \ \mathcal{I}_{\lambda_n}(u_n) = c_{\lambda_n}.
 \end{equation}
 Then it follows from \eqref{Q00} and Lemma \ref{lem 4.2} that $\mathcal{P}_{\lambda_n}(u_n)=0$.
 From \eqref{Ilu}, \eqref{X31}, \eqref{Plt}, \eqref{Q00} and Lemma \ref{lem 4.4} (iii), one has
 \begin{eqnarray}\label{Q01}
   C_4
    & \ge & c_{\lambda_n}=\mathcal{I}_{\lambda_n}(u_n)-\frac{1}{N+\alpha}\mathcal{P}_{\lambda_n}(u_n)\nonumber\\
    &  =  & \frac{2+\alpha}{2(N+\alpha)}\|\nabla u_n\|_2^2+\frac{1}{2(N+\alpha)}\int_{\R^N}\left[\alpha V(x)
             -\nabla V(x)\cdot x\right]u_n^2\mathrm{d}x\nonumber\\
    & \ge & \frac{\gamma_3}{2(N+\alpha)}\|u_n\|^2.
 \end{eqnarray}
 This shows that $\{\|u_n\|\}$  is bounded in $H^1(\R^N)$.
 In view of Lemma \ref{lem 4.4} (v), we have $\lim_{n\to\infty}c_{\lambda_n}=c_*\le c_1$. Hence,
 it follows from \eqref{Ilu} and \eqref{Q00} that
 \begin{equation}\label{Q04}
   \mathcal{I}(u_n)\rightarrow  c_*, \ \ \ \  \mathcal{I}'(u_n)\rightarrow 0.
 \end{equation}
 This shows that $\{u_n\}$ satisfies \eqref{PS} with $c_{\lambda}=c_*$. In view of the proof of Lemma \ref{lem 4.8}, we can show that
 there exists $\bar{u}\in H^1(\R^N)\setminus \{0\}$ such that \eqref{Q05} holds.
 \end{proof}

 \begin{proof}[Proof of Theorem  \ref{thm1.6}] Let $\hat{m}:=\inf_{u\in\mathcal{K}}\mathcal{I}(u)$.
 Then Lemma \ref{lem 4.9} shows that $\mathcal{K}\ne \emptyset$ and $\hat{m}\le c_1$. For any $u\in \mathcal{K}$, Lemma \ref{lem 4.2}
 implies $\mathcal{P}(u)=\mathcal{P}_1(u)=0$. Hence it follows from \eqref{D33} that $\mathcal{I}(u)=\mathcal{I}_1(u)>0$ for all
 $u\in \mathcal{K}$, and so $\hat{m}\ge 0$.
 Let $\{u_n\}\subset \mathcal{K}$ such that
 \begin{equation}\label{Z01}
   \mathcal{I}'(u_n)=0, \ \ \ \ \mathcal{I}(u_n) \rightarrow \hat{m}.
 \end{equation}
 In view of Lemma \ref{lem 4.5}, $\hat{m}\le c_1<m_1^{\infty}$. By a similar argument as in the proof of Lemma \ref{lem 4.8}, we can prove
 that there exists $\bar{u}\in H^1(\R^N)\setminus \{0\}$ such that
 \begin{equation}\label{Z02}
   \mathcal{I}'(\bar{u})=0, \ \ \ \ \mathcal{I}(\bar{u}) = \hat{m}.
 \end{equation}
 This shows that $\bar{u}$ is a least energy solution of \eqref{SE}.
 \end{proof}

 \vskip6mm
 {\section{Semiclassical states for \eqref{KE9}}}
 \setcounter{equation}{0}

 \vskip2mm
 \par
   In this section, we give the proof of Theorem \ref{thm1.9}. From now on we assume without loss of generality that $x_0 = 0$, that is
 $V(0)<V_{\infty}$. Performing the scaling $u(x) = v(\varepsilon x)$ one easily sees that problem \eqref{KE9} is equivalent to
 \begin{equation}\label{KE10}
 \left\{
   \begin{array}{ll}
     -\triangle u+V_{\varepsilon}(x)u=(I_{\alpha}*F(u))f(u), & x\in \R^N; \\
     u\in H^1(\R^N),
   \end{array}
 \right.
 \end{equation}
 where $V_{\varepsilon}(x)=V(\varepsilon x)$. The energy functional associated to problem \eqref{KE10} is given by
 \begin{equation}\label{Iv}
   \mathcal{I}^{\varepsilon}(u)=\frac{1}{2}\int_{\R^N}\left(|\nabla u|^2+V_{\varepsilon}(x)u^2\right)\mathrm{d}x
            -\frac{1}{2}\int_{\R^N}(I_{\alpha}*F(u))F(u)\mathrm{d}x.
 \end{equation}
 As in Section 3, we also define, for $\lambda\in [1/2, 1]$ and $\varepsilon\ge 0$, the family of functionals
 $\mathcal{I}_{\lambda}^{\varepsilon} : H^1(\R^N) \rightarrow \R$ as follows
 \begin{equation}\label{Ilv}
   \mathcal{I}_{\lambda}^{\varepsilon}(u)=\frac{1}{2}\int_{\R^N}\left(|\nabla u|^2+V_{\varepsilon}(x)u^2\right)\mathrm{d}x
            -\frac{\lambda}{2}\int_{\R^N}(I_{\alpha}*F(u))F(u)\mathrm{d}x.
 \end{equation}
 Since $V\in \mathcal{C}(\R^N, \R)$, $V(0)< V_{\infty}$ and $u_1^{\infty}
 \in H^1(\R^N)\setminus \{0\}$, then there exist $\hat{r}>0$ and $R_0>0$ such that
 \begin{equation}\label{W40}
    V_{\infty}-V(x)>\frac{1}{4}\left(V_{\infty}-V(0)\right), \ \ \ \ \forall \ |x|\le \hat{r},
 \end{equation}
 \begin{equation}\label{W42}
    \left[V_{\infty}-V(0)+4\cdot 3^N\left(V_{\max}-V_{\infty}\right)\right]
                \int_{|x|> R_0}|u_1^{\infty}|^2\mathrm{d}x\le \frac{1}{2}\left(V_{\infty}-V(0)\right)\|u_1^{\infty}\|_2^2
 \end{equation}
 and
 \begin{eqnarray}\label{W43}
   &     & T^N\left(V_{\max}-V_{\infty}\right)\int_{|x|> R_0}|u_1^{\infty}|^2\mathrm{d}x\nonumber\\
   & \le & \frac{\min\{\mathfrak{g}(1/2),\mathfrak{g}(3/2)\}\|\nabla u_1^{\infty}\|_2^2
             +V_{\infty}\min\{\mathfrak{h}(1/2),\mathfrak{h}(3/2)\}\|u_1^{\infty}\|_2^2}{2(N+\alpha)}.
 \end{eqnarray}

 \par
   Similar to Lemma \ref{lem 4.4}, we can prove the following lemma.

 \begin{lemma}\label{lem 5.1}
 Assume that {\rm(V1)} and {\rm(F1)-(F3)} hold. Then
 \begin{enumerate}[{\rm(i)}]
  \item $\mathcal{I}_{\lambda}^{\varepsilon}\left((u_1^{\infty})_{T}\right)<0$
 for all $\lambda\in [0.5, 1]$ and $\varepsilon\ge 0$;
  \item there exists a positive constant $\hat{\kappa}_0 $ independent of $\lambda$ and $\varepsilon\ge 0$ such that for all
  $\lambda\in [0.5, 1]$ and $\varepsilon\ge 0$,
 \begin{equation*}
   c_{\lambda}^{\varepsilon}:=\inf_{\gamma\in \Gamma}\max_{t\in [0, 1]}\mathcal{I}_{\lambda}^{\varepsilon}(\gamma(t))\ge \hat{\kappa}_0
     >\max\left\{\mathcal{I}_{\lambda}^{\varepsilon}(0), \mathcal{I}_{\lambda}^{\varepsilon}\left((u_1^{\infty})_{T}\right)\right\},
 \end{equation*}
 where $\Gamma$ is defined by \eqref{Ga};
 \item $c_{\lambda}^{\varepsilon}$ is bounded for $\lambda\in [0.5, 1]$ and $\varepsilon\ge 0$.
 \end{enumerate}
\end{lemma}

 \begin{lemma}\label{lem 5.2}
 Assume that {\rm(V1), (V5)} and {\rm(F1)-(F3)} hold. Then there exists $\tilde{\lambda}\in [1/2, 1)$ such that $c_{\lambda}^{\varepsilon}<m_{\lambda}^{\infty}$ for $\lambda\in (\tilde{\lambda}, 1]$
 and $\varepsilon\in [0, \varepsilon_0]$, where and in the sequel $\varepsilon_0:=\hat{r}/R_0T$.
 \end{lemma}

 \begin{proof} For any $\varepsilon\ge 0$, it is easy to see that $\mathcal{I}_{\lambda}^{\varepsilon}\left((u_1^{\infty})_t\right)$
 is continuous on $t\in (0, \infty)$. Hence for any $\lambda\in [1/2, 1]$ and $\varepsilon\ge 0$,  we can choose
 $t_{\lambda}^{\varepsilon}\in (0, T)$ such that $\mathcal{I}_{\lambda}^{\varepsilon} \left((u_1^{\infty})_{t_{\lambda}^{\varepsilon}}\right)
 =\max_{t\in [0,T]}\mathcal{I}_{\lambda}^{\varepsilon}\left((u_1^{\infty})_t\right)$. Setting $\gamma_0(t)$ as in \eqref{ga0}.
 Then $\gamma_0\in \Gamma$ defined by \eqref{Ga}. Moreover
 \begin{equation}\label{W50}
   \mathcal{I}_{\lambda}^{\varepsilon} \left((u_1^{\infty})_{t_{\lambda}^{\varepsilon}}\right)
    =\max_{t\in [0,1]}\mathcal{I}_{\lambda}^{\varepsilon}\left(\gamma_0(t)\right)\ge c_{\lambda}^{\varepsilon}.
 \end{equation}
 Since $\mathcal{P}^{\infty}(u_1^{\infty})=0$, then $\int_{\R^N}(I_{\alpha}*F(u_1^{\infty}))F(u_1^{\infty})\mathrm{d}x>0$.
 Let
 \begin{eqnarray}\label{W59}
   \tilde{\lambda}
    & :=  & \min\left\{\frac{1}{2}, 1-\frac{\left(V_{\infty}-V(0)\right)\|u_1^{\infty}\|_2^2}{8\cdot 3^{N}
             \int_{\R^N}(I_{\alpha}*F(u_1^{\infty}))F(u_1^{\infty})\mathrm{d}x}, \right.\nonumber\\
    &     & \ \ \ \ \ \ \left. 1-\frac{\min\{\mathfrak{g}(1/2),\mathfrak{g}(3/2)\}\|\nabla u_1^{\infty}\|_2^2
             +V_{\infty}\min\{\mathfrak{h}(1/2),\mathfrak{h}(3/2)\}\|u_1^{\infty}\|_2^2}
             {2(N+\alpha)T^{N}\int_{\R^N}(I_{\alpha}*F(u_1^{\infty}))F(u_1^{\infty})\mathrm{d}x}\right\}. \ \ \ \
 \end{eqnarray}
 Then it follows from \eqref{B10}, \eqref{B20} and (V5) that $1/2\le \tilde{\lambda}<1$. We have two cases to distinguish:

 \vskip2mm
 \par
   Case i). $t_{\lambda}^{\varepsilon}\in [1/2, 3/2]$. From \eqref{Il}, \eqref{G45}, \eqref{Ilv}-\eqref{W59} and Lemma \ref{lem 4.4} (iv),
  we have
 \begin{eqnarray*}
   m_{\lambda}^{\infty}
    & \ge & m_1^{\infty}=\mathcal{I}_1^{\infty}(u_1^{\infty})\ge \mathcal{I}_1^{\infty}\left((u_1^{\infty})_{t_{\lambda}^{\varepsilon}}\right)\nonumber\\
    &  =  & \mathcal{I}_{\lambda}^{\varepsilon}\left((u_1^{\infty})_{t_{\lambda}^{\varepsilon}}\right)
              -\frac{(1-\lambda)(t_{\lambda}^{\varepsilon})^{N}}{2}\int_{\R^N}(I_{\alpha}*F(u_1^{\infty}))F(u_1^{\infty})\mathrm{d}x\nonumber\\
              &       &   +\frac{(t_{\lambda}^{\varepsilon})^N}{2}\int_{\R^N}[V_{\infty}
              -V_{\varepsilon}(t_{\lambda}^{\varepsilon}x)]|u_1^{\infty}|^2\mathrm{d}x\nonumber\\
    & \ge & c_{\lambda}^{\varepsilon} -\frac{3^{N}(1-\lambda)}{2^{N+1}}\int_{\R^N}(I_{\alpha}*F(u_1^{\infty}))F(u_1^{\infty})\mathrm{d}x\nonumber\\
    &     & \ \     +\frac{V_{\infty}-V(0)}{2^{N+3}}\int_{|x|\le R_0}|u_1^{\infty}|^2\mathrm{d}x
                    -\frac{3^N\left(V_{\max}-V_{\infty}\right)}{2^{N+1}}\int_{|x|> R_0}|u_1^{\infty}|^2\mathrm{d}x\nonumber\\
    &  =  & c_{\lambda}^{\varepsilon} -\frac{3^{N}(1-\lambda)}{2^{N+1}}\int_{\R^N}(I_{\alpha}*F(u_1^{\infty}))F(u_1^{\infty})\mathrm{d}x
             +\frac{V_{\infty}-V(0)}{2^{N+3}}\|u_1^{\infty}\|_2^2\nonumber\\
    &     & \ \ -\frac{V_{\infty}-V(0)+4\cdot 3^N\left(V_{\max}-V_{\infty}\right)}{2^{N+3}}
                \int_{|x|> R_0}|u_1^{\infty}|^2\mathrm{d}x\nonumber\\
    & \ge & c_{\lambda}^{\varepsilon} -\frac{3^{N}(1-\lambda)}{2^{N+1}}\int_{\R^N}(I_{\alpha}*F(u_1^{\infty}))F(u_1^{\infty})\mathrm{d}x
             +\frac{V_{\infty}-V(0)}{2^{N+4}}\|u_1^{\infty}\|_2^2\nonumber\\
    &  >  & c_{\lambda}^{\varepsilon}, \ \ \ \ \forall \ \lambda\in (\tilde{\lambda}, 1], \ \ \varepsilon\in [0, \varepsilon_0].
 \end{eqnarray*}
 \par
   Case ii). $t_{\lambda}^{\varepsilon}\in (0, 1/2)\cup (3/2, T)$. From \eqref{Il}, \eqref{G45}, \eqref{Ilv}-\eqref{W59} and Lemma
 \ref{lem 4.4} (iv), we have
 \begin{eqnarray*}
   m_{\lambda}^{\infty}
    & \ge & m_1^{\infty}=\mathcal{I}_1^{\infty}(u_1^{\infty})\ge \mathcal{I}_1^{\infty}\left((u_1^{\infty})_{t_{\lambda}^{\varepsilon}}\right)
             +\frac{\mathfrak{g}(t_{\lambda}^{\varepsilon})\|\nabla u_1^{\infty}\|_2^2
             +V_{\infty}\mathfrak{h}(t_{\lambda}^{\varepsilon})\|u_1^{\infty}\|_2^2}{2(N+\alpha)}\nonumber\\
    &  =  & \mathcal{I}_{\lambda}^{\varepsilon}\left((u_1^{\infty})_{t_{\lambda}^{\varepsilon}}\right)
              -\frac{(1-\lambda)(t_{\lambda}^{\varepsilon})^{N}}{2}\int_{\R^N}(I_{\alpha}*F(u_1^{\infty}))F(u_1^{\infty})\mathrm{d}x\nonumber\\
    &     & \ \  +\frac{(t_{\lambda}^{\varepsilon})^N}{2}\int_{\R^N}[V_{\infty}-V_{\varepsilon}(t_{\lambda}^{\varepsilon}x)]
              |u_1^{\infty}|^2\mathrm{d}x +\frac{\mathfrak{g}(t_{\lambda}^{\varepsilon})\|\nabla u_1^{\infty}\|_2^2
             +V_{\infty}\mathfrak{h}(t_{\lambda}^{\varepsilon})\|u_1^{\infty}\|_2^2}{2(N+\alpha)}\nonumber\\
    & \ge & c_{\lambda}^{\varepsilon} -\frac{(1-\lambda)T^{N}}{2}\int_{\R^N}(I_{\alpha}*F(u_1^{\infty}))F(u_1^{\infty})\mathrm{d}x
              -\frac{T^N\left(V_{\max}-V_{\infty}\right)}{2}\int_{|x|> R_0}|u_1^{\infty}|^2\mathrm{d}x\nonumber\\
    &     & \ \  +\frac{\min\{\mathfrak{g}(1/2),\mathfrak{g}(3/2)\}\|\nabla u_1^{\infty}\|_2^2
                 +V_{\infty}\min\{\mathfrak{h}(1/2),\mathfrak{h}(3/2)\}\|u_1^{\infty}\|_2^2}{2(N+\alpha)}\nonumber\\
    & \ge & c_{\lambda}^{\varepsilon} -\frac{(1-\lambda)T^{N}}{2}\int_{\R^N}(I_{\alpha}*F(u_1^{\infty}))F(u_1^{\infty})\mathrm{d}x\nonumber\\
    &     & \ \ +\frac{\min\{\mathfrak{g}(1/2),\mathfrak{g}(3/2)\}\|\nabla u_1^{\infty}\|_2^2
             +V_{\infty}\min\{\mathfrak{h}(1/2),\mathfrak{h}(3/2)\}\|u_1^{\infty}\|_2^2}{4(N+\alpha)}\nonumber\\
    &  >  & c_{\lambda}^{\varepsilon}, \ \ \ \ \forall \ \lambda\in (\tilde{\lambda}, 1], \ \ \varepsilon\in [0, \varepsilon_0].
 \end{eqnarray*}
 In both cases, we obtain that $c_{\lambda}^{\varepsilon}<m_{\lambda}^{\infty}$ for $\lambda\in (\tilde{\lambda}, 1]$
 and $\varepsilon\in [0, \varepsilon_0]$.
 \end{proof}

 \begin{lemma}\label{lem 5.3}
 Assume that {\rm(V1), (V5), (V6)} and {\rm(F1)-(F3)} hold. Then for every $\varepsilon\in (0, \varepsilon_0]$ and for almost every
 $\lambda\in (\tilde{\lambda},1]$, there exists $u_{\lambda}^{\varepsilon}\in H^1(\R^N)\setminus \{0\}$ such that
 \begin{equation}\label{DD31}
   (\mathcal{I}_{\lambda}^{\varepsilon})'(u_{\lambda}^{\varepsilon})=0, \ \ \ \ \mathcal{I}_{\lambda}^{\varepsilon}(u_{\lambda}^{\varepsilon}) = c_{\lambda}^{\varepsilon}.
 \end{equation}
 \end{lemma}

 \begin{proof}  For any fixed $\varepsilon\in (0, \varepsilon_0]$, under (V1) and (F1)-(F3), Lemma \ref{lem 5.1} implies that $\mathcal{I}_{\lambda}^{\varepsilon}(u)$ satisfies the assumptions of Proposition \ref{pro 4.1} with $X=H^1(\R^N)$, $J=[\tilde{\lambda},1]$
 and $\Phi_{\lambda}=\mathcal{I}_{\lambda}^{\varepsilon}$. So for almost every $\lambda\in (\tilde{\lambda},1]$, there exists a bounded
 sequence $\{u_n^{\varepsilon}(\lambda)\} \subset H^1(\R^N)$ (for simplicity, we denote the sequence by $\{u_n^{\varepsilon}\}$ instead of $\{u_n^{\varepsilon}(\lambda)\}$) such that
 \begin{equation}\label{PSv}
   \mathcal{I}_{\lambda}^{\varepsilon}(u_n^{\varepsilon})\rightarrow c_{\lambda}^{\varepsilon}>0,
   \ \ \ \ (\mathcal{I}_{\lambda}^{\varepsilon})'(u_n^{\varepsilon}) \rightarrow 0.
 \end{equation}
 By Lemma \ref{lem 4.6}, there exist a subsequence of $\{u_n^{\varepsilon}\}$, still denoted by
 $\{u_n^{\varepsilon}\}$, and $u_{\lambda}^{\varepsilon}\in H^1(\R^N)$, an integer $l\in \N\cup \{0\}$, and $w^1, \ldots, w^l\in H^1(\R^N)\setminus \{0\}$ such that
 \begin{equation}\label{Un1}
   u_n^{\varepsilon}\rightharpoonup u_{\lambda}^{\varepsilon}\ \ \mbox {in} \  H^1(\R^N),
    \ \ \ \  (\mathcal{I}_{\lambda}^{\varepsilon})'(u_{\lambda}^{\varepsilon})=0,
 \end{equation}
 \begin{equation}\label{Un2}
   (\mathcal{I}_{\lambda}^{\infty})'(w^k)=0, \ \ \ \ \mathcal{I}_{\lambda}^{\infty}(w^k)\ge m_{\lambda}^{\infty},\ \ \ \ 1\le k\le l
 \end{equation}
 and
 \begin{equation}\label{Bb1}
    c_{\lambda}^{\varepsilon}= \mathcal{I}_{\lambda}^{\varepsilon}(u_{\lambda}^{\varepsilon})+\sum_{k=1}^{l}\mathcal{I}_{\lambda}^{\infty}(w^k).
 \end{equation}

 \par
 Since $(\mathcal{I}_{\lambda}^{\varepsilon})'(u_{\lambda}^{\varepsilon})=0$, then it follows from Lemma \ref{lem 4.2} that
 \begin{eqnarray}\label{Plv}
  \mathcal{P}_{\lambda}^{\varepsilon}(u_{\lambda}^{\varepsilon})
    &  :=  & \frac{N-2}{2}\|\nabla u_{\lambda}^{\varepsilon}\|_2^2
             +\frac{1}{2}\int_{\R^N}\left[NV_{\varepsilon}(x)+\nabla V_{\varepsilon}(x)\cdot x\right]
             (u_{\lambda}^{\varepsilon})^2\mathrm{d}x\nonumber\\
    &     & \ \  -N\lambda\int_{\R^N}(I_{\alpha}*F(u_{\lambda}^{\varepsilon}))F(u_{\lambda}^{\varepsilon})\mathrm{d}x=0.
 \end{eqnarray}
 Since $\|u_n^{\varepsilon}\|\nrightarrow 0$, we deduce from \eqref{Un2} and \eqref{Bb1} that if $u_{\lambda}=0$ then $l\ge 1$ and
 \begin{eqnarray*}
   c_{\lambda}^{\varepsilon} =  \mathcal{I}_{\lambda}^{\varepsilon}(u_{\lambda}^{\varepsilon})+\sum_{k=1}^{l}\mathcal{I}_{\lambda}^{\infty}(w^k)
    \ge  m_{\lambda}^{\infty},
 \end{eqnarray*}
 which contradicts with Lemma \ref{lem 5.2}. Thus $u_{\lambda}^{\varepsilon}\ne 0$.
 It follows from \eqref{Ilv}, \eqref{Plv} and (V6) that
 \begin{eqnarray}\label{DD33}
   \mathcal{I}_{\lambda}^{\varepsilon}(u_{\lambda}^{\varepsilon})
     &  =  & \mathcal{I}_{\lambda}^{\varepsilon}(u_{\lambda}^{\varepsilon})
              -\frac{1}{N+\alpha}\mathcal{P}_{\lambda}^{\varepsilon}(u_{\lambda}^{\varepsilon})\nonumber\\
     &  =  & \frac{2+\alpha}{2(N+\alpha)}\|\nabla u_{\lambda}^{\varepsilon}\|_2^2
               +\frac{1}{2(N+\alpha)}\int_{{\R}^N}\left[\alpha V_{\varepsilon}(x)-\nabla V_{\varepsilon}(x)\cdot x\right](u_{\lambda}^{\varepsilon})^2\mathrm{d}x\nonumber\\
     & \ge & \frac{1}{2(N+\alpha)}\left[(2+\alpha)\|\nabla u_{\lambda}^{\varepsilon}\|_2^2+(1-\theta'')\alpha V(0)
              \|u_{\lambda}^{\varepsilon}\|_2^2\right]>0.
 \end{eqnarray}
 From \eqref{Bb1} and \eqref{DD33}, one has
 \begin{eqnarray}\label{DD34}
   c_{\lambda}^{\varepsilon} =  \mathcal{I}_{\lambda}^{\varepsilon}(u_{\lambda}^{\varepsilon})+\sum_{k=1}^{l}\mathcal{I}_{\lambda}^{\infty}(w^k)
    >  lm_{\lambda}^{\infty}.
 \end{eqnarray}
 By Lemma \ref{lem 5.2}, we have $c_{\lambda}^{\varepsilon}<m_{\lambda}^{\infty}$ for $\lambda\in (\tilde{\lambda}, 1]$, which, together with
 \eqref{DD34}, implies that $l=0$ and $\mathcal{I}_{\lambda}^{\varepsilon}(u_{\lambda}^{\varepsilon}) = c_{\lambda}^{\varepsilon}$.
 \end{proof}

 \begin{lemma}\label{lem 5.4}
 Assume that {\rm(V1), (V5), (V6)} and {\rm(F1)-(F3)} hold. Then for any $\varepsilon\in (0,\varepsilon_0]$, there exists
 $\bar{u}^{\varepsilon}\in H^1(\R^N)\setminus \{0\}$ such that $(\mathcal{I}^{\varepsilon})'(\bar{u}^{\varepsilon})=0$ and $\mathcal{I}^{\varepsilon}(\bar{u}^{\varepsilon})>0$.
 \end{lemma}

 \begin{proof} In view of Lemma \ref{lem 5.3}, for any fixed $\varepsilon\in (0, \varepsilon_0]$, there exist two sequences
 $\{\lambda_n\}\subset [\tilde{\lambda}, 1]$ and $\{u_{\lambda_n}^{\varepsilon}\}\subset H^1(\R^N)\setminus \{0\}$, denoted by
 $\{u_n^{\varepsilon}\}$, such that
 \begin{equation}\label{R00}
   \lambda_n\rightarrow 1, \ \ \ \ c_{\lambda_n}^{\varepsilon}\rightarrow c_*^{\varepsilon},
   \ \ \ \ (\mathcal{I}_{\lambda_n}^{\varepsilon})'(u_n^{\varepsilon})=0,
    \ \ \ \ 0<\mathcal{I}_{\lambda_n}^{\varepsilon}(u_n^{\varepsilon}) = c_{\lambda_n}^{\varepsilon}.
 \end{equation}
 Then it follows from \eqref{R00} and Lemma \ref{lem 4.2} that $\mathcal{P}_{\lambda_n}^{\varepsilon}(u_n^{\varepsilon})=0$.
 From  (V6), \eqref{Ilv}, \eqref{Plv}, \eqref{R00} and Lemma \ref{lem 5.1} (iii), one has
 \begin{eqnarray}\label{QQ01}
   C_6
    & \ge & c_{\lambda_n}^{\varepsilon}=\mathcal{I}_{\lambda_n}^{\varepsilon}(u_n^{\varepsilon})
              -\frac{1}{N+\alpha}\mathcal{P}_{\lambda_n}^{\varepsilon}(u_n^{\varepsilon})\nonumber\\
    &  =  & \frac{2+\alpha}{2(N+\alpha)}\|\nabla u_n^{\varepsilon}\|_2^2
              +\frac{1}{2(N+\alpha)}\int_{\R^N}\left[\alpha V_{\varepsilon}(x)
              -\nabla V_{\varepsilon}(x)\cdot x\right](u_n^{\varepsilon})^2\mathrm{d}x\nonumber\\
    & \ge & \frac{1}{2(N+\alpha)}\left[(2+\alpha)\|\nabla u_{\lambda}^{\varepsilon}\|_2^2+(1-\theta'')\alpha V(0)
              \|u_{\lambda}^{\varepsilon}\|_2^2\right].
 \end{eqnarray}
 This shows that $\{\|u_n^{\varepsilon}\|\}$  is bounded in $H^1(\R^N)$.
 In view of \eqref{R00}, we have $\lim_{n\to\infty}c_{\lambda_n}^{\varepsilon}=c_*^{\varepsilon}$. Hence,
 it follows from \eqref{Iv} and \eqref{R00} that
 \begin{equation*}
   \mathcal{I}^{\varepsilon}(u_n^{\varepsilon})\rightarrow  c_*^{\varepsilon}, \ \ \ \  (\mathcal{I}^{\varepsilon})'(u_n^{\varepsilon})\rightarrow 0.
 \end{equation*}
 This shows that $\{u_n^{\varepsilon}\}$ satisfies \eqref{PSv} with $c_{\lambda}^{\varepsilon}=c_*^{\varepsilon}$. In view of the proof
 of Lemma \ref{lem 5.3}, we can show that there exists $\bar{u}^{\varepsilon}\in H^1(\R^N)\setminus \{0\}$ such that $(\mathcal{I}^{\varepsilon})'(\bar{u}^{\varepsilon})=0$ and $\mathcal{I}^{\varepsilon}(\bar{u}^{\varepsilon})>0$.
 \end{proof}

 \begin{proof}[Proof of Theorem  \ref{thm1.9}] By a similar argument as the proof of Theorem  \ref{thm1.6}, we can prove
 Theorem  \ref{thm1.9} by using Lemmas \ref{lem 5.2}, \ref{lem 5.3} and \ref{lem 5.4} instead of \ref{lem 4.5}, \ref{lem 4.8}
 and \ref{lem 4.9}, respectively, so, we omit it.
 \end{proof}

 \section*{Acknowledgements}
This work was partially supported by the National Natural Science Foundation of China (11571370).

\end{document}